\documentclass[a4paper,12pt,reqno]{amsart}
\usepackage{amsfonts}
\usepackage{amsmath}
\usepackage{amssymb}
\usepackage[a4paper]{geometry}
\usepackage{mathrsfs}
\usepackage[colorlinks]{hyperref}
\renewcommand\eqref[1]{(\ref{#1})} %Need with hyperref
\numberwithin{equation}{section}
\theoremstyle{plain}
\newtheorem{thm}{Theorem}[section]
\newtheorem{proposition}[thm]{Proposition}
\newtheorem{cor}[thm]{Corollary}
\newtheorem{lemma}[thm]{Lemma}

\theoremstyle{definition}
\newtheorem{defn}[thm]{Definition}

\newcommand{\Rn}{\mathbb R^{n}}
\newcommand{\Zn}{{\mathbb Z^{n}}}
\newcommand{\Tn}{{\mathbb T^{n}}}
\newcommand{\Op}{{\textrm{\rm Op}}}
\newcommand{\Optn}{{\textrm{\rm Op}_\Tn}}
\newcommand{\Opzn}{{\textrm{\rm Op}_\Zn}}
\newcommand{\HS}{{\mathtt{HS}}}

\def\Ftn{{\mathcal F}_\Tn}
\def\Fzn{{\mathcal F}_\Zn}

\def\p#1{{\left({#1}\right)}}

\newcommand{\R}{\mathbb{R}}
\newcommand{\C}{\mathbb{C}}
\begin{document}

   \title[Difference equations and pseudo-differential operators on $\mathbb{Z}^n$]
   {Difference equations and pseudo-differential operators on $\mathbb{Z}^n$}

%\title[Difference equations and pseudo-difference operators]
%   {Difference equations and pseudo-difference operators}

\author[Linda N. A. Botchway]{Linda N. A. Botchway}
\address{
  Linda N. A. Botchway:
  \endgraf
  African Institute for Mathematical Sciences
  \endgraf
  AIMS-GH, Biriwa
  \endgraf
  Ghana
   \endgraf
  {\it E-mail address} {\rm linda@aims.edu.gh}
  }
\author[P. Ga\"el Kibiti]{P. Ga\"el Kibiti}
\address{
  P. Ga\"el Kibiti:
  \endgraf
  African Institute for Mathematical Sciences
  \endgraf
  AIMS-GH, Biriwa
  \endgraf
  Ghana
   \endgraf
  {\it E-mail address} {\rm gael@aims.edu.gh}
  }
  \author[Michael Ruzhansky]{Michael Ruzhansky}
\address{
  Michael Ruzhansky:
  \endgraf
	Department of Mathematics: Analysis, Logic and Discrete Mathematics
	\endgraf
	Ghent University, Belgium
	\endgraf
	and
	\endgraf
	School of Mathematical Sciences
		\endgraf Queen Mary University of London 
			\endgraf
		United Kingdom
		\endgraf
	and
			\endgraf
			 Department of Mathematics
  \endgraf
  Imperial College London
  \endgraf
  United Kingdom
  \endgraf
	{\it E-mail address} {\rm Michael.Ruzhansky@ugent.be}  }

\thanks{The third author was supported in parts by the EPSRC
 grants EP/K039407/1, EP/R003025/1, by the Leverhulme Grants RPG-2014-02, RPG-2017-151, and by the FWO Odysseus 1 Grant G.0H94.18N: Analysis and Partial Differential Equations.}

     
  \dedicatory{Dedicated to the $85^{th}$ birthday of Francis Kofi Ampenyi Allotey}   

     \begin{abstract}
     In this paper we develop the calculus of pseudo-differential operators on the lattice $\mathbb{Z}^n$, which we can call {\em pseudo-difference operators}. An interesting feature of this calculus is that the global frequency space ($\mathbb{T}^n$) is compact so the symbol classes are defined in terms of the behaviour with respect to the lattice variable. We establish formulae for composition, adjoint, transpose, and for parametrix for the elliptic operators. We also give conditions for the $\ell^2$, weighted $\ell^2$, and $\ell^p$ boundedness of operators and for their compactness on $\ell^p$. We describe a link to the toroidal quantization on the torus $\mathbb{T}^n$, and apply it to give conditions for the membership in Schatten classes on $\ell^2(\mathbb{Z}^n)$. Furthermore, we discuss a version of Fourier integral operators on the lattice and give conditions for their $\ell^2$-boundedness. The results are applied to give estimates for solutions to difference equations on the lattice $\mathbb{Z}^n$. Moreover, we establish G{\aa}rding  and sharp G{\aa}rding inequalities, with an application to the unique solvability of parabolic equations on  the lattice $\mathbb{Z}^n$.
     \end{abstract}
     \maketitle

\tableofcontents

\section{Introduction}

The aim of this paper is to develop a calculus of pseudo-differential operators suitable for the applications to solving difference equations on the lattice $\mathbb{Z}^n$. Such equations naturally appear in various problems of modelling and in the discretisation of continuous problems. 
We call the appearing operators pseudo-difference operators.

As a simple motivating example, consider the equation
  \begin{equation}\label{EQ:ex0}
  \sum_{j=1}^{n}\Big(f(k + v_j) - f(k - v_j)\Big) + af(k)=g(k), \quad k\in\Zn,
  \end{equation}
 with  $v_{j} = (0, \ldots, 0, 1, 0, \ldots, 0)\in \mathbb{Z}^n$, where the $j^{th}$ element of $v_j$ is  $1$, and all other elements are $0$. The idea of this paper is to use the suitable Fourier analysis for solving difference equations of this type. Thus, if, for example, ${\rm Re}\, a\not=0$, this equation is solvable for any $g\in\ell^2(\Zn)$ and the solution can be given by the formula
\begin{equation}\label{EQ:ex1}
f(k) = \int_{\mathbb{T}^n}e^{2\pi k\cdot x}\frac{1}{2i\sum_{j=1}^n \sin(2\pi x_j) + a} \widehat{g}(x)\text{d}x,
\end{equation} 
    where 
      \begin{equation}\label{gael}
    \widehat{g}(x) = \sum_{k\in \mathbb{Z}^{n}} e^{-2\pi i k\cdot x} g(k),\quad x\in\Tn,
      \end{equation}
 is the Fourier transform of $g$. 
 Formula \eqref{EQ:ex1} also extends to give solutions to \eqref{EQ:ex0} for any tempered growth function $g\in\mathcal{S}'(\Zn)$.
 In particular, if $g\in\ell^2(\Zn)$ then the solution $f$ to the difference equation \eqref{EQ:ex0} given by \eqref{EQ:ex1} satisfies $f\in\ell^2(\Zn)$ 
  and, more generally,
 if $g$ satisfies $$
   \sum_{k\in\Zn} (1+|k|)^{s} |g(k)|^2<\infty
 $$
 for some $s\in\R$, then the 
 solution $f$ to the difference equation \eqref{EQ:ex0} given by \eqref{EQ:ex1} also satisfies
 $$
   \sum_{k\in\Zn} (1+|k|)^{s} |f(k)|^2<\infty,
 $$
see Example (3) in Section \ref{SEC:Ex}.
 
 \smallskip
 From the point of view of the theory of pseudo-differential operators the operators of the form \eqref{EQ:ex1} extend the usual difference operators on the lattice, thus we feel that the term {\em pseudo-difference operators} may be justified to emphasise that they extend the class of difference operators into a $*$-algebra. This agrees with the terminology already existing in the literature (see e.g. \cite{Rab09}).
 
 \smallskip
 The theory of pseudo-differential operators is usually effective in answering a number of questions such as:

\begin{itemize}
\item What kind of difference equations, similar to \eqref{EQ:ex0}, are solvable in this way?
\item Given $g(k)$, what are properties of $f(k)$ given the representation formula \eqref{EQ:ex1}?
\item What about variable coefficient versions of difference equations, where the coefficients of the equations may also depend on $k$?
\end{itemize} 

It is the purpose of this paper to answer these and other questions by developing a suitable theory of pseudo-differential operators on the lattice $\Zn$. There are several interesting features of this theory making it essentially different from the classical theory of pseudo-differential operators on $\Rn$, such as

\begin{itemize}
\item The phase space is $\Zn\times\Tn$ with the frequencies being elements of the compact space $\Tn$ (the torus $\Tn:=\Rn/\Zn$). The usual theory of pseudo-differential operators works with symbol classes with increasing decay of symbols after taking their derivatives in the frequency variable. Here we can not expect any improving decay properties in frequency since the frequency space is compact.
\item We can not work with derivatives with respect to the space variable $k\in \Zn$. Therefore, this needs to be replaced by working with appropriate difference operators on the lattice.
\end{itemize} 

The developed theory is similar in spirit to the global theory of (toroidal) pseudo-differential operators on the torus $\Tn$ consistently developed in \cite{RT-JFAA}, see also
\cite{Agran79,Agran84,Amosov88} as well as \cite{RT-Birk,RT-NFAO} for earlier works. In particular, symbol classes in this paper will coincide with symbol classes developed in \cite{RT-JFAA,ruzhansky2009pseudo} but with a twist, swapping the order of the space and frequency variables. As a result, we can draw on properties of these symbol classes developed in the above works. Several attempts of developing a suitable theory of pseudo-differential operators on the lattice $\Zn$ have been done in the literature, see e.g. \cite{Rab2010,Rab09}, but with no symbolic calculus. Operators on the one-dimensional lattice $\mathbb{Z}$ have been considered in 
\cite{molahajloo2009pseudo,DW2013,ghaemi2015study}, but again with no symbolic calculus, and $\ell^p$ estimates were considered in \cite{CR2011} and \cite{Catana14}. 
There are numerous physical models realised as difference equations, see e.g. \cite{RR-2006,Rab09,Rab13} for the analysis of Schr\"odinger, Dirac, and other operators on lattices, and their spectral properties. 
 
Our symbol classes exhibit improvement when differences are taken with respect to the space (lattice) variable, thus resembling in their behaviour the so-called SG pseudo-differential operators in $\Rn$, developed by Cordes \cite{Cordes-bk}, but again with a twist in variables.

In the recent work \cite{Mantoiu-Ruzhansky-DM}, a framework has been developed for the theory of pseudo-differential operators on general locally compact type I groups, with application to spectral properties of operators. The Kohn-Nirenberg type quantization formula that the analysis of this paper relies on makes a special case of the construction of \cite{Mantoiu-Ruzhansky-DM}, but there is only limited symbolic calculus available there due to the generality of the setting. Thus, here we are able to provide much deeper analysis in terms of the asymptotic expansions and formulae for the appearing symbols and kernels.

Compared to situations when the state space is compact (for example, \cite{RT-IMRN} on compact groups or \cite{RN-IMRN} on compact manifolds) the calculus here is essentially different since we can not construct it using standard methods relying on the decay properties in the frequency component of the phase space since the frequency space is our case is the torus $\Tn$ which is  compact, so no improvement with respect to the decay of the frequency variable is possible.

To give some further details, the Fourier transform of $f\in \ell^{1}(\mathbb{Z}^n)$  is defined by 
     \begin{equation}\label{gael1}
    \mathcal{F}_{\Zn}f(x):= \widehat{f}(x) := \sum_{k\in \mathbb{Z}^{n}} e^{-2\pi i k\cdot x}
     f(k),
     \end{equation}
     for $ x\in \mathbb{T}^n=\Rn/\Zn$, where we will be denoting, throughout the paper,  
     $$k\cdot x = \sum_{j=1}^{n}k_{j}x_{j},$$   where $k = (k_{1},\ldots, k_{n})$ and $x = (x_{1}, \ldots, x_{n})$.
     The Fourier transform extends to $\ell^{2}(\mathbb{Z}^n)$ and the constants are normalised in such a way that we have the Plancherel formula  
     \begin{equation}\label{perrin}
     \sum_{k\in \mathbb{Z}^{n}}|f(k)|^{2} = \int_{\mathbb{T}^n}|\widehat{f}(x)|^{2}\text{d}x.
     \end{equation}
     The Fourier inversion formula takes the form
     \begin{equation}\label{EQ:Finv}
     f(k) = \int_{\mathbb{T}^n}e^{2\pi i k\cdot x}\widehat{f}(x)\text{d}x, \ \ \ k\in \mathbb{Z}^n.
     \end{equation}
     For a measurable function  
     $\sigma :\mathbb{Z}^{n}\times\mathbb{T}^{n}\rightarrow \mathbb{C}$,  we define the sequence $\text{Op}(\sigma)f$ by
            \begin{equation}\label{neil}
                    \text{Op}(\sigma)f(k) := \int_{\mathbb{T}^n}e^{2\pi i k\cdot x}\sigma(k,x)\widehat{f}(x)\text{d}x, \ \ \ k\in \mathbb{Z}^n.
                    \end{equation}
            The operator defined by equation \eqref{neil}
             will be called the pseudo-differential operator on $\mathbb{Z}^n$ corresponding to the symbol $\sigma=\sigma(k,x)$, $(k,x)\in \mathbb{Z}^{n}\times\mathbb{T}^{n}.$ 
       We will also call it a {\em pseudo-difference operator} and the quantization $\sigma\mapsto\Op(\sigma)$ the {\em lattice quantization}.
             
The Schwartz space $\mathcal{S}(\mathbb{Z}^n)$ on the lattice $\Zn$ is the space of rapidly decreasing functions $\varphi: \mathbb{Z}^n\rightarrow \mathbb{C}$, that is, $\varphi\in \mathcal{S}(\mathbb{Z}^n)$ if for any $M<\infty$ there exits a constant $C_{\varphi, M}$ such that
   \[|\varphi(k)| \leq C_{\varphi, M}(1+|k|)^{-M},\quad \textrm{for all}\;k\in \mathbb{Z}^n.\]
The topology on $\mathcal{S}(\mathbb{Z}^n)$ is given by the seminorms $p_{j}$, where $j\in \mathbb{N}_{0}$ and $p_{j}(\varphi): = \sup\limits_{k\in \mathbb{Z}^n} (1+|k|)^{j}|\varphi(k)|.$
The space of tempered distributions $\mathcal{S}'(\mathbb{Z}^n)$ is the topological dual to $\mathcal{S}(\mathbb{Z}^n)$, i.e. the space of all linear continuous functionals on $\mathcal{S}(\mathbb{Z}^n)$.

  As usual, the theory of pseudo-differential operators applies not only to specific class of operators but to general linear continuous operators on the space. Indeed, let 
  $A: \ell^\infty(\mathbb{Z}^n)\rightarrow \mathcal{S}'(\mathbb{Z}^n)$ be a continuous linear operator. Then  it can be shown that $A$ can be written in the form $A=\Op(\sigma)$ with
   the symbol $\sigma=\sigma(k,x)$ defined by
    \[\sigma(k, x):= e_{-x}(k)Ae_x(k) = e^{-2\pi ik\cdot x}A\Big(e^{2\pi ik\cdot x}\Big),\]
where $e_x(k)= e^{2\pi ik\cdot x}$ for all $k\in\Zn$ and $x\in \mathbb{T}^n.$ Indeed, using the Fourier inversion formula \eqref{EQ:Finv}
in the usual way one can justify the simple calculation
  \begin{equation*}
\begin{aligned}
Af(k) & = A\Big( \int_{\mathbb{T}^n}e^{2\pi i k\cdot x}\widehat{f}(x)\text{d}x\Big) \\
& =  \int_{\mathbb{T}^n}A\left(e^{2\pi  i k\cdot x}\right)\widehat{f}(x)\text{d}x \\
& =  \int_{\mathbb{T}^n} e^{2\pi  i k\cdot x}\sigma(k,x) \widehat{f}(x)\text{d}x  = \Op(\sigma)f(k).
\end{aligned}
\end{equation*} 

We also present the following applications of the developed calculus:

\begin{itemize}
\item conditions for $\ell^2(\Zn)$-boundedness, compactness, and membership in Schatten-von Neumann classes  for operators in terms of their symbols; Gohberg lemma and estimates for the essential spectrum of operators;
\item conditions for weighted $\ell^2(\Zn)$-boundedness and weighted a-priori estimates for difference equations;
\item Fourier series operators and their $\ell^2(\Zn)$-boundedness;
\item G{\aa}rding  and sharp G{\aa}rding inequalities, with an application to the unique solvability of parabolic equations on  the lattice $\mathbb{Z}^n$.
\end{itemize} 
We also present conditions for $\ell^p(\Zn)$-boundedness and compactness, extending results of \cite{molahajloo2009pseudo} and \cite{CR2011}. We can note that compared to the existing literature on $\ell^2$-boundedness, our results do not require any decay properties of the symbol, thus also leading to a-priori estimates for elliptic difference equations without any loss of decay.
  
    In Section \ref{SEC:symbols} we introduce symbol classes and discuss the kernels of the corresponding pseudo-difference operators. An interesting difference with the usual theory of pseudo-differential operators is that since the space $\Zn$ is discrete, the Schwartz kernels of the corresponding pseudo-difference operators do not have singularity at the diagonal.

The plan of the paper is as follows. We study the properties of pseudo-difference operator on $\mathbb{Z}^n$  by 
first discussing in Section \ref{SEC:symbols} their symbols and kernels, as well as amplitudes.
The symbolic calculus is developed in Section \ref{SEC:calculus}.
In Section \ref{SEC:link} we establish the link between the quantizations on the lattice $\Zn$ and the torus $\Tn$. 
In Section \ref{SEC:L2-HS} we investigate the boundedness on $\ell^2(\Zn)$, weighted $\ell^2(\Zn)$, $\ell^p(\Zn)$, compactness on $\ell^p(\Zn)$, and give conditions for the membership in Schatten-von Neumann classes.  Finally, in Section \ref{SEC:Ex} we give some examples.

Throughout the paper we will use the notation $\mathbb{N}_0=\mathbb{N}\cup\{0\}.$

\medskip
{\bf Acknowledgements.} 
  The authors would like to thank AIMS Ghana and its academic director Emmanuel Essel for the hospitality during the first two authors' study there and during the third author's visits to Ghana and to the African Institute for Mathematical Sciences (AIMS) when this work was carried out.
  The authors would also like to thank Julio Delgado for discussions and valuable remarks.
  
 \section{Symbols, kernels, and amplitudes}
\label{SEC:symbols}
   
   For the developing of the symbolic calculus and for the definition of the symbol classes we need to have some analogues of derivatives in the space variable. For this purpose, we will be using the following difference operators.
      
    \begin{defn}[Difference operators]
     We define $\Delta^{\alpha}$ acting on functions $\tau:\Zn\to\C$ by the formula
     \begin{equation}\label{EQ:diffs}
\Delta^{\alpha} \tau (k):=\int_{\mathbb{T}^n} e^{2\pi i k\cdot y} (e^{2\pi i  y} -1 )^\alpha  \widehat{\tau}(y)\text{d}y, 
\end{equation} 
     where $\alpha=(\alpha_1,\ldots,\alpha_n)$ and 
     $$
     (e^{2\pi i  y} -1 )^\alpha=(e^{2\pi i  y_1} -1 )^{\alpha_1}\cdots (e^{2\pi i  y_n} -1 )^{\alpha_n}.
     $$
     It is easy to see that we have the decomposition 
     \begin{equation}\label{EQ:diffs0} 
     \Delta^{\alpha} = \Delta^{\alpha_1}_{1}\cdot \ldots \cdot \Delta^{\alpha_n}_{n},
     \end{equation} 
      where denoting $v_j = (0, \cdots ,0, 1,0, \cdots ,0)$ with $1$ at the $j^{th}$ position,  we have
     \begin{eqnarray}\label{EQ:diffs2}
     \Delta_j \tau (k) &=&\int_{\mathbb{T}^n} e^{2\pi i (k+v_j)\cdot y}  \widehat{\tau}(y)\  \text{d}y - \int_{\mathbb{T}^n} e^{2\pi i k\cdot y}  \widehat{\tau}(y)\text{d}y \nonumber \\
     &=& \tau (k + v_j) - \tau (k)
     \end{eqnarray}
    are the usual difference operators on $\Zn$. Therefore, formulae \eqref{EQ:diffs0}-\eqref{EQ:diffs2} give the equivalent characterisation to  \eqref{EQ:diffs}, and can be taken as the definition of difference operators $\Delta^{\alpha}$.
    
    At the same time, the representation \eqref{EQ:diffs} becomes useful for comparing operators \eqref{EQ:diffs0}-\eqref{EQ:diffs2} to more general difference operators that will be introduced in Definition \ref{DEF:gdiff}.
    \end{defn}
    
    The formula \eqref{EQ:diffs} makes sense for $\tau\in{\mathcal S}'(\Zn)$. Indeed, in this case we have $\widehat{\tau}\in\mathcal{D}'(\Tn)$ and the formula \eqref{EQ:diffs} can be interpreted in terms of the distributional duality on $\Tn$, 
    \begin{equation}\label{EQ:ddif}
 \Delta^{\alpha} \tau (k)=\langle \widehat{\tau},e^{2\pi i k\cdot y} (e^{2\pi i  y} -1 )^\alpha\rangle
\end{equation} 
    acting on the $y$-variable.
    These operators have been introduced, analysed and shown to satisfy many useful properties, such as the Leibniz formula, summation by parts formula, Taylor expansion formula, and many others, in \cite{RT-JFAA} and
    \cite[Section 3.3]{ruzhansky2009pseudo} to which we refer for detailed discussions.
    
    As usual, we will be using the notation 
    $$ D_x^\alpha= D_{x_1}^{\alpha_1}\cdots D_{x_n}^{\alpha_n},\quad
    D_{x_j}=\frac{1}{2\pi i}\frac{\partial}{\partial x_j}.$$
    
    In the sequel we will be also using the multi-index notation $\alpha!=\alpha_1!\cdots\alpha_n!$
    
   It will be also convenient to use operators
   \begin{equation}\label{EQ:das}
D_x^{(\alpha)}= D_{x_1}^{(\alpha_1)}\cdots D_{x_n}^{(\alpha_n)},\quad
D_{x_j}^{(\ell)}=\prod_{m=0}^\ell \left(\frac{1}{2\pi i}\frac{\partial}{\partial x_j} -m\right), \quad \ell\in\mathbb{N}.
\end{equation} 
As usual, $D_x^0=D_x^{(0)}=I$. The operators $D_x^{(\alpha)}$ become very useful in the analysis related to the torus as they appear in the periodic Taylor expansion, see \eqref{EQ:Taylor}. Their precise form in \eqref{EQ:das} is related to properties of Stirling numbers, see \cite[Section 3.4]{ruzhansky2009pseudo}.
    
    \begin{defn}[Symbol classes $S^\mu_{\rho,\delta} (\mathbb{Z}^n \times \mathbb{T}^n)$]  
    \label{DEF:symbolclasses}  
    Let $\rho,\delta\in\R$.
    We say that a function $\sigma:\mathbb{Z}^n \times \mathbb{T}^n\to\C$ belongs to $S^\mu_{\rho,\delta} (\mathbb{Z}^n \times \mathbb{T}^n)$ if $\sigma(k, \cdot) \in C^\infty ( \mathbb{T}^n) $ for all $k \in \mathbb{Z}^n$, and for all multi-indices $\alpha,\beta$ there exists a positive constant $C_{\alpha,\beta}$ such that we have
    \begin{equation}\label{EQ:symbdef}
|D^{(\beta)}_{x}\Delta^{\alpha}_{k} \sigma(k,x)| \leq C_{\alpha,\beta}(1 +|k|)^{\mu-\rho|\alpha|+\delta|\beta|}
\end{equation} 
for all $k\in\Zn$ and $x\in\Tn$. 

If $\rho=1$ and $\delta=0$, we will denote simply 
$S^\mu (\mathbb{Z}^n \times \mathbb{T}^n):=S^\mu_{1,0} (\mathbb{Z}^n \times \mathbb{T}^n).$

We denote by $\Op(\sigma)$ the operator with symbol $\sigma$ given by 
    \begin{equation}\label{neil2}
                    \text{Op}(\sigma)f(k) := \int_{\mathbb{T}^n}e^{2\pi i k\cdot x}\sigma(k,x)\widehat{f}(x)\text{d}x, \ \ \ k\in \mathbb{Z}^n,
                    \end{equation}
and by $\Op(S^{\mu}_{\rho, \delta}(\mathbb{Z}^n\times \mathbb{T}^n))$ the collection of operators $\Op(\sigma)$ as $\sigma$ varies over the symbol class $S^\mu_{\rho,\delta} (\mathbb{Z}^n \times \mathbb{T}^n)$.
    \end{defn}
    
   Here and everywhere we may often write $\Delta^{\alpha}=\Delta^{\alpha}_{k}$ to emphasise that the difference operators are acting on functions with respect to the variable $k$.
   We note that these symbol classes, modulo swapping the order of the variables $x$ and $k$, have been extensively analysed and used in \cite{RT-JFAA} for the development of the global toroidal calculus of pseudo-differential operators on the torus $\Tn$. We also refer to 
    \cite[Chapter 4]{ruzhansky2009pseudo} for a thorough presentation of their properties.
    
   The classes on the torus, similar to Definition \ref{DEF:symbolclasses}, were analysed in \cite{ruzhansky2009pseudo}, and their equivalence (also on general compact Lie groups) to the usual H\"ormander classes was shown in \cite{ruzhansky2014hormander}.
          
   Pseudo-differential operator can be represented in various forms. For example, for suitable functions $f$, using formula \eqref{gael1} we can write
   \begin{eqnarray*}
   \text{Op}(\sigma)f(k) &=& \int_{\mathbb{T}^n}e^{2\pi ik\cdot x}\sigma(k,x) \widehat{f}(x)\text{d}x\\
   &=& \int_{\mathbb{T}^n} \sum_{m \in \mathbb{Z}^n}e^{2\pi i(k-m)\cdot x}\sigma(k,x) f(m)\text{d}x\\
   &=& \sum_{m \in \mathbb{Z}^n}\int_{\mathbb{T}^n}e^{2\pi i(k-m)\cdot x}\sigma(k,x) f(m)\text{d}x \\
   &=& \sum_{l \in \mathbb{Z}^n}\int_{\mathbb{T}^n} e^{2\pi il\cdot x}\sigma(k,x) f(k-l)\text{d}x \\
   &=& \sum_{l \in \mathbb{Z}^n} \kappa(k,l) f(k-l) \\
   &=& \sum_{m \in \mathbb{Z}^n} K(k,m) f(m),
   \end{eqnarray*}
   with kernels
   \begin{equation}\label{EQ:kernels}
K(k,m)= \kappa(k,k-m) \quad \text{and}  \quad \kappa(k,l)= \int_{\mathbb{T}^n}e^{2\pi il\cdot x}\sigma(k,x)\text{d}x. 
\end{equation} 
   
   We now establish some properties of the kernels of pseudo-difference operators on $\mathbb{Z}^n$ with symbols $\sigma\in S^{\mu}_{\rho, \delta}(\mathbb{Z}^n\times \mathbb{T}^n)$.
   
   \begin{thm}\label{THM:kernel}
   Let $\sigma\in S^{\mu}_{\rho, \delta}(\mathbb{Z}^n\times \mathbb{T}^n)$ and let $\delta\geq 0$. Then for every $N\in\mathbb{N}_0$ there exists a positive constant $C_N>0$ such that we have
   
   \begin{equation}\label{EQ:kernelsprops}
   |K(k, m)|\leq C_N (1+|k|)^{\mu + 2N\delta} (1+|k-m|)^{-2N}, 
   \end{equation}
   for all $k, m\in\Zn$.
   \end{thm}
   
   In particular we note that in comparison to pseudo-differential operators on $\Rn$ or on $\Tn$, the kernel $K(k,m)$ is well defined for $k=m$ and has no singularity at the diagonal since the space $\Zn\times\Zn$ is discrete. We also note that we do not need any further restrictions on $\rho$ and $\delta$ in Theorem \ref{THM:kernel}.
   
   \begin{proof}[Proof of Theorem \ref{THM:kernel}]
   We note that for $k=m$ we have, using \eqref{EQ:kernels}, that
   \begin{equation}\label{EQ:kerdiag}
K(k,k)=\kappa(k, 0) = \int_{\mathbb{T}^n}\sigma(k, x)\text{d}x,
\end{equation} 
satisfying \eqref{EQ:kernelsprops} in this case.
   
   Let us now assume that $k\not=m$, so that also $l=k-m\not=0$.
   Denoting the Laplacian on $\Tn$ by $\displaystyle \mathcal{L}_x := \sum_{j = 1}^{n}\frac{\partial ^2}{\partial x_j^2}$, we have 
   \begin{equation}\label{EQ:Lapl}
(1 - \mathcal{L}_x)e^{2\pi il\cdot x} = \big(1 + 4\pi ^{2} |l|^{2}\big)e^{2\pi il\cdot x} \ \ \ ; \ \ \ e^{2\pi il\cdot x} = \frac{(1 - \mathcal{L}_x)}{1 + 4\pi ^{2} |l|^{2}}  e^{2\pi il\cdot x},
\end{equation}  
   so that for $l\neq 0$ we can write
   \begin{eqnarray*}
   \kappa(k, l) &=& \int_{\mathbb{T}^n}e^{2\pi il\cdot x}\sigma(k, x)\text{d}x\\
   &=& \int_{\mathbb{T}^n}\Bigg(\frac{(1 - \mathcal{L}_x)^N}{\big(1 + 4\pi ^{2} |l|^{2}\big)^N}e^{2\pi il\cdot x}\Bigg)\sigma(k, x)\text{d}x\\
   &=& (1 + 4\pi ^{2} |l|^{2})^{-N}\int_{\mathbb{T}^n}e^{2\pi i l\cdot x}\big(1 - \mathcal{L}_x\big)^N\sigma(k, x)\text{d}x.
   \end{eqnarray*}
     Therefore,  for all $N\geq 0$ we have
   \[|\kappa(k,l)|\leq  C_N (1 + 4\pi ^{2} |l|^{2})^{-N}(1+|k|)^{\mu + 2N\delta}.\]
   It follows then from \eqref{EQ:kernels} that $K(k,m)$ satisfies \eqref{EQ:kernelsprops}.
       \end{proof}
       
      Similar to the classical cases, we have the formula extracting the symbol from an operator.
       
        \begin{proposition}\label{PROP:symbols}
  The symbol of a pseudo-difference operator $A$ % \in \Op(S^{\mu}_{\rho, \delta}(\mathbb{Z}^n\times \mathbb{T}^n))$ 
  is given by 
  \begin{equation}\label{B}
   \sigma(k,x) = e^{-2\pi i k\cdot x}Ae_x(k),
   \end{equation}
   where $e_x (k) = e^{2\pi i k\cdot x},$ for all $k\in\Zn$ and $x\in\Tn$.
  \end{proposition}
  
  \begin{proof}
  For the function $e_y (l) = e^{2\pi i l\cdot y}$, its Fourier transform is given formally by 
   \[
   \widehat{ e_y} (x)  = \sum_{l \in \mathbb{Z}^n} e^{-2\pi i l\cdot x}e^{2\pi i l\cdot y},
   \]
with the usual justification in terms of limits or distributions.
 Plugging this into the formula  
    \begin{equation*}
   \text{Op}(\sigma)f(k) = \int_{\mathbb{T}^n} e^{2\pi i k\cdot x} \sigma(k,x) \widehat{f}(x)\text{d}x,
   \end{equation*}
      it follows that 
   \begin{eqnarray*}
   \text{Op}(\sigma)e_y(k) &=& \int_{\mathbb{T}^n} \sum_{l \in \mathbb{Z}^n}e^{2\pi i k\cdot x} \sigma(k,x) e^{-2\pi i l\cdot x}e^{2\pi i l\cdot y}\text{d}x\\
   &=& \int_{\mathbb{T}^n} \sum_{l \in \mathbb{Z}^n}e^{-2\pi i (l-k)\cdot x} \sigma(k,x) e^{2\pi i l\cdot y} \text{d}x\\
   &=&  \sum_{l \in \mathbb{Z}^n}\widehat{ \sigma}(k,l-k) e^{2\pi i l\cdot y} \\
   &=&  \sum_{m \in \mathbb{Z}^n}\widehat{ \sigma}(k,m)e^{2\pi i m\cdot y} e^{2\pi i k\cdot y}  \qquad (\text{where} \ l-k=m)\\
   &=& \sigma(k,y)e^{2\pi i k\cdot y},
   \end{eqnarray*}
   where $\widehat{\sigma}$ stands for the Fourier transform on $\Tn$ in the second variable, and where we used the toroidal Fourier inversion formula by a standard distributional interpretation.
   This gives formula \eqref{B}.
  \end{proof}

   From the definition \eqref{neil} of pseudo-differential operators and writing out the Fourier transform of $f$ using formula \eqref{gael1} gives the amplitude representation of pseudo-difference operators as    
   \begin{equation}\label{3}
   \text{Op}(\sigma)f(k) = \sum_{m \in \mathbb{Z}^n}\int_{\mathbb{T}^n}e^{2\pi i(k-m)\cdot x}\sigma(k,x) f(m)\text{d}x.
   \end{equation}
   
   This motivates analysing amplitude operators of the form
     \begin{equation}\label{AMP}
   Af(k) = \sum_{m \in \mathbb{Z}^n} \int_{\mathbb{T}^n} e^{2\pi i (k-m)\cdot x} a(k,m,x)f(m)\text{d}x,
   \end{equation}
with amplitudes $a:\Zn\times\Zn\times\Tn\to\C$. We may still denote such operators by $\Op(a)$, which is consistent with \eqref{3}. 
   
     \begin{defn}[Amplitude classes $\mathcal{A}^{\mu_1, \mu_2}_{\rho, \delta} (\mathbb{Z}^n \times \mathbb{T}^n)$]
     Let $\rho,\delta\in\R$.
   A function $a:\Zn\times\Zn\times\Tn\to\C$ is said to belong to the amplitude class 
   $\mathcal{A}^{\mu_1,\mu_2}_{\rho,\delta} (\mathbb{Z}^n \times \mathbb{Z}^n \times \mathbb{T}^n) $  if $a(k,m, \cdot) \in C^\infty (  \mathbb{T}^n) $ for all $k,m \in \mathbb{Z}^n$,  and if  for all multi-indices $\alpha,\beta,\gamma$ there exists a positive constant $C_{\alpha,\beta ,\gamma}>0$ such that for some $J\in\mathbb{N}_0$ with $J\leq |\gamma|$ we have
   \begin{equation}\label{EQ:amps}
|D^{(\gamma)}_{y}\Delta^{\alpha}_{k}\Delta^{\beta}_{m} a(k,m,y)| \leq C_{\alpha,\beta,\gamma}(1 +|k |)^{\mu_1-\rho|\alpha| + \delta J}(1 +|m |)^{\mu_2-\rho|\beta| + \delta(|\gamma|-J)}.
\end{equation}  
   We note that clearly $\displaystyle S_{\rho, \delta}^{\mu} ( \mathbb{Z}^n \times \mathbb{T}^n) \subset \mathcal{A}^{\mu,0}_{\rho, \delta} (\mathbb{Z}^n \times \mathbb{Z}^n \times \mathbb{T}^n).$ 
   The space of amplitude operators $\Op(a)$ with amplitudes $a\in \mathcal{A}^{\mu_1,\mu_2}_{\rho,\delta} (\mathbb{Z}^n \times \mathbb{Z}^n \times \mathbb{T}^n) $ will be denoted by 
   $\Op(\mathcal{A}^{\mu_1,\mu_2}_{\rho,\delta} (\mathbb{Z}^n \times \mathbb{Z}^n \times \mathbb{T}^n) ).$
   \end{defn}
The definition above is motivated by properties of symbols in Definition \ref{DEF:symbolclasses}, by the property that the amplitude of the operator adjoint to $\Op(\sigma)$ will be given by 
$a(k,m,x)=\overline{\sigma(m,x)}$, and in order to have Theorem \ref{THM:amplitudes}.
Here, for the inclusion $\displaystyle S_{\rho, \delta}^{\mu} ( \mathbb{Z}^n \times \mathbb{T}^n) \subset \mathcal{A}^{\mu,0}_{\rho, \delta} (\mathbb{Z}^n \times \mathbb{Z}^n \times \mathbb{T}^n)$ we may take $J=|\gamma|$ in \eqref{EQ:amps}, while for the amplitude $a(k,m,x)=\overline{\sigma(m,x)}$ we may take $J=0$.

   We now aim to show that  $$\Op(\mathcal{A}^{\mu_1,\mu_2}_{\rho,\delta} (\mathbb{Z}^n \times \mathbb{Z}^n \times \mathbb{T}^n) )\subset \Op(S^{\mu_1+\mu_2}_{\rho, \delta}(\mathbb{Z}^n\times \mathbb{T}^n)).$$ For this, we establish a useful property of more general difference operators. 
      
       \begin{defn}[Generalised difference operators]\label{DEF:gdiff} Let $q\in C^\infty(\Tn)$. Then for $\tau:\Zn\to\C$ we 
          define the $q$-difference operator by 
          \begin{equation}\label{deltaoperator12}
          \Delta_{q}\tau(k):= \int_{\mathbb{T}^n}e^{2\pi i k\cdot x}q(x)\widehat{\tau}(x)\text{d}x.
          \end{equation}
          While the integral formula above makes sense for suitable functions $\tau$, similar to
          \eqref{EQ:ddif} it can be extended to all $\tau\in \mathcal{S}'(\Zn)$ by the distributional duality
\begin{equation}\label{EQ:ddif2}
 \Delta_q \tau (k)=\langle \widehat{\tau},e^{2\pi i k\cdot x} q(x)\rangle
\end{equation} 
    acting on the $x$-variable. At the same time, formula \eqref{deltaoperator12} also extends to non-smooth functions $q$: for example, \eqref{deltaoperator12} makes sense for $\tau\in\ell^2(\Zn)$ and $q\in L^2(\Tn)$, or for other choices of matching conditions on $\tau$ and $q$, for \eqref{deltaoperator12} to make sense.
          \end{defn}
      Expanding $\widehat{\tau}(x)$ we can also note the useful formula 
          \begin{equation}\label{deltaoperator2}
           \Delta_{q}\tau(k)= \sum_{l\in \mathbb{Z}^n}\int_{\mathbb{T}^n}e^{2\pi i (k-l)\cdot x}q(x)\tau(l)\text{d}x= \sum_{l\in \mathbb{Z}^n} \tau(l) \mathcal{F}_{\Zn}^{-1}q(k-l)=(\tau*\mathcal{F}_{\Zn}^{-1}q)(k).
           \end{equation}
         
    We record the following property of generalised difference operators acting on symbols.
    
   \begin{lemma}\label{lemma1}
   Let $0\leq \delta\leq 1$ and let $\sigma\in S_{\rho,\delta}^{\mu}(\mathbb{Z}^n\times \mathbb{T}^n)$, $\mu\in\R$. Then for any $q\in C^{\infty}(\mathbb{T}^n)$ and any $\beta\in \mathbb{N}_0^{n} $ we have  
   \begin{equation}\label{EQ:lemma1}
   |\Delta_{q}D_{x}^{(\beta)}\sigma(k, x)|\leq  C_{q,\beta}(1+|k|)^{\mu +\delta|\beta| }, 
   \end{equation}
  for all $k\in \mathbb{Z}^n$ and $x\in \mathbb{T}^n$.
  %, with $\displaystyle M = \Big[\frac{n + \max\{0, \mu\}}{2}\Big] + 1.$
   \end{lemma}
   \begin{proof}
   It is enough to prove this for $\beta = 0$. Using \eqref{deltaoperator2}, we write $ \Delta_q\sigma(k, x)$ as
   \begin{eqnarray*}
  \Delta_q\sigma(k, x) &=& \sum_{l \in \mathbb{Z}^n}\int_{\mathbb{T}^n}e^{2\pi i(k-l)\cdot y}q(y)\sigma(l, x)\text{d}y \\
   &=&  \sigma(k, x)\int_{\mathbb{T}^n}q(y)\text{d}y  +  \sum_{\substack{
   l\in \mathbb{Z}^n\\
   l\neq k}}\int_{\mathbb{T}^n}e^{2\pi i(k-l)\cdot y}q(y)\sigma(l, x)\text{d}y \\
   &=:& I_1 + I_2,
   \end{eqnarray*} 
   where in the first term we set $l=k$. Then we have
   $$|I_1|\leq  (1+|k|)^{\mu }.$$ 
   On the other hand, for $\mu\geq 0$, integrating by parts with the operator \eqref{EQ:Lapl}, we have
 \begin{equation}\label{EQ:seri}
\begin{aligned}
  |I_2| &= \left|\sum_{\substack{
   l\in \mathbb{Z}^n\\ l\neq k}}\int_{\mathbb{T}^n}\frac{e^{2\pi i(k-l)\cdot y}}{(2\pi)^{2M}|k-l|^{2M}}\Bigg(\mathcal{L}_{y}^{M}q(y)\Bigg)\sigma(l, x)\text{d}y\right| \\
   &\leq  C\sum_{\substack{
   l\in \mathbb{Z}^n\\
   l\neq k}}\frac{1}{|k-l|^{2M}}(1+|l|)^{\mu}\\
   &\leq  C\sum_{m\neq 0}\frac{1}{|m|^{2M}}(1+|k-m|)^{\mu}\\
   &\leq  C\sum_{m\neq 0}\frac{1}{|m|^{2M}}\Bigg((1+|k|)^{\mu } + |m|^{\mu} \Bigg)\\
   &\leq  C(1+|k|)^{\mu},
\end{aligned}
\end{equation} 
   where we used that $\mu \geq 0$ in the last lines and that
   if we take  $\displaystyle M > \frac{n + \mu}{2}$, then $2M - \mu > n,$ and the series in the last lines of \eqref{EQ:seri} converges.
   
   If $\mu< 0$, we will use the Peetre inequality which says that for all $s\in\R$ and $\xi,\eta\in\Rn$ we have
   $$
   (1+|\xi+\eta|)^s\leq 2^{|s|}(1+|\xi|)^s (1+|\eta|)^{|s|},
   $$
   see \cite[Proposition 3.3.31]{ruzhansky2009pseudo}. Applying this with $s=\mu$, we have
   $$
   (1+|k-m|)^\mu\leq 2^{|\mu|}(1+|k|)^\mu (1+|m|)^{|\mu|}.
   $$
   Applying this to the third line of \eqref{EQ:seri} we get that
   $$
   |I_2|\leq C(1+|k|)^{\mu},
   $$
   provided that we take $M$ such that $2M-|\mu|>n$, so that the series in $m$ converges.
   So we obtain \eqref{EQ:lemma1} in all the cases.
   \end{proof} 
   
   Before proving that amplitude operators are pseudo-difference operators and are given by symbols, let us recall the periodic Taylor expansion formula from \cite[Theorem 3.4.4]{ruzhansky2009pseudo}. It says that if $h\in C^\infty(\Tn)$ then we have the periodic Taylor expansion for $h$ given by
   \begin{equation}\label{EQ:Taylor}
h(x)=\sum_{|\alpha|<N} \frac{1}{\alpha !} (e^{2\pi i x}-1)^\alpha D_z^{(\alpha)} h(z)|_{z=0} +
\sum_{|\alpha|=N} h_\alpha(x) (e^{2\pi i x}-1)^\alpha,
\end{equation} 
where $h_\alpha\in C^\infty(\Tn)$, with $D_z^{(\alpha)}$ given by
\eqref{EQ:das}, and
$$
     (e^{2\pi i  x} -1 )^\alpha=(e^{2\pi i  x_1} -1 )^{\alpha_1}\cdots (e^{2\pi i  x_n} -1 )^{\alpha_n}.
     $$
The functions  $h_\alpha\in C^\infty(\Tn)$ are products of one-dimensional functions $h_j(y)$, $y\in\mathbb{T}$, defined inductively by $h_0=h$ and then by
      \begin{equation}\label{EQ:hjs}
h_{j + 1}(y):= \left\{ \begin{aligned}
         \frac{h_j(y) - h_j(0)}{e^{2\pi i  y} -1}  \ \ \text{if}\ \ y \neq 0, \\
          D_y h_j(y) \ \text{if} \ y = 0. \end{aligned}\right.
\end{equation} 
In particular, we have
      \[h_{1}(y) = \left\{ \begin{aligned}
      \frac{h(y) - h(0)}{e^{2\pi i  y} -1} \ \text{if}\ \ y \neq 0, \\
       D_y h(0) \ \text{if}  \ y = 0,\end{aligned}\right.\]
      \[h_{2}(y) = \left\{ \begin{aligned}
       & &\frac{h_1(y) - h_1(0)}{e^{2\pi i  y} -1} = \frac{\frac{h(y) - h(0)}{e^{2\pi i  y} -1} -  D_y h(0)}{e^{2\pi i  y} -1} \ \text{if}\ \ y \neq 0, \\
      & & D_y h_1(0) \ \text{if} \ y = 0,
       \end{aligned}\right.\]
 and so on. It can be shown that these functions are smooth everywhere, including at $y=0$, and that $h_j$ depends on the $j^{th}$ order derivative of $h$.     
We refer to \cite[Section 3.4]{ruzhansky2009pseudo} for the proof as well as for the expressions and analysis for the remainder functions $h_\alpha$.
   
   \begin{thm}\label{THM:amplitudes}
   Let $0\leq \delta<\rho \leq 1$. Let $a\in \mathcal{A}^{\mu_1,\mu_2}_{\rho,\delta} (\mathbb{Z}^n \times \mathbb{Z}^n \times \mathbb{T}^n)$ and let the corresponding amplitude operator $A$ be given by
   \begin{equation}\label{A}
   Af(k) = \sum_{l \in \mathbb{Z}^n} \int_{\mathbb{T}^n} e^{2\pi i (k-l)\cdot x} a(k,l,x)f(l)\text{\rm d}x.
   \end{equation}
   Then we have $A = \Op(\sigma_A)$ for some $\sigma_A \in S^{\mu_1+\mu_2}_{\rho,\delta} (\mathbb{Z}^n \times  \mathbb{T}^n) $. Moreover, $\sigma_A$ has the following asymptotic expansion
   \begin{equation}\label{EQ:expamp}
\sigma_A (k,x) \sim \sum_{\alpha } \frac{1}{\alpha !}\Delta_{l}^{\alpha} D_{x}^{(\alpha)} a (k,l,x)\Big|_{l=k},
\end{equation} 
   which means  that for all $N\in\mathbb{N}$ we have
    \begin{equation}\label{EQ:expamp2}
    \sigma_A - \sum_{|\alpha| < N } \frac{1}{\alpha !}\Delta_{l}^{\alpha} D_{x}^{(\alpha)} a (k,l,x)\Big|_{l=k}  \in  S^{\mu_1+\mu_2 -N(\rho-\delta)}_{\rho,\delta} (\mathbb{Z}^n \times  \mathbb{T}^n).
    \end{equation} 
   \end{thm}
   
   From now on we will always understand asymptotic sums of type \eqref{EQ:expamp} in the sense of \eqref{EQ:expamp2}.
   
   \begin{proof}[Proof of Theorem \ref{THM:amplitudes}]
   Using formula \eqref{B} from Proposition \ref{PROP:symbols}, the symbol of the operator $A$ in \eqref{A} is given by
   \begin{eqnarray*}
   \sigma_A (k,x) &=& e^{-2\pi i k\cdot x}\sum_{l \in \mathbb{Z}^n}\int_{\mathbb{T}^n} e^{2\pi i (k-l)\cdot y} a(k,l,y) e^{2\pi i l\cdot x} \text{d}y\\
   &=& \sum_{l \in \mathbb{Z}^n}\int_{\mathbb{T}^n} e^{2\pi i (k-l)\cdot (y-x)} a(k,l,y)\text{d}y\\
   &=& \int_{\mathbb{T}^n} e^{2\pi i k\cdot (y-x)} \widehat{ a}(k,y-x,y)\text{d}y,
   \end{eqnarray*}
  where $\widehat{a}$ denotes the Fourier transform of $a$ with respect to the second variable.
  Replacing $y$ by $y+x$, we obtain
   \begin{equation}
   \label{C}
   \sigma_A (k,x)  = \int_{\mathbb{T}^n} e^{2\pi i k\cdot y} \widehat{ a}(k,y,y+x)\text{d}y.
   \end{equation}
   Taking the Taylor expansion of $\widehat{ a}(k,y,y+x)$ in the third variable at $x$ as in \eqref{EQ:Taylor}, we have
   \begin{equation}
   \label{D}
   \widehat{a}(k,y,x+y) = \sum_{|\alpha| \leq  N} \frac{1}{\alpha !}(e^{2\pi i  y} -1 )^\alpha D_{x}^{(\alpha)} \widehat{a}(k,y,x) + R_0, 
   \end{equation}
   where $R_0$ is a remainder that we will study later.
   Substituting \eqref{D} into \eqref{C} gives 
   \begin{equation}
   \label{E}
   \sigma_A (k,x)  =\int_{\mathbb{T}^n} e^{2\pi i k\cdot y} \sum_{|\alpha| \leq N} \frac{1}{\alpha !}(e^{2\pi i  y} -1 )^\alpha D_{x}^{(\alpha)} \widehat{a}(k,y,x)\text{d}y + R,
   \end{equation}
   with the remainder $R$ that can be expressed in terms of $R_0$.
  Since by \eqref{EQ:diffs} we have
  
   \[\Delta^{\alpha}_{l} \tau (k)=\int_{\mathbb{T}^n} e^{2\pi i k\cdot y} (e^{2\pi i  y} -1 )^\alpha  \widehat{\tau}(y)\text{d}y, \]
  the formula \eqref{E} becomes
   \[ \sigma_A (k,x) = \sum_{|\alpha|\leq  N} \frac{1}{\alpha !}\Delta^{\alpha}_{l} D_{x}^{(\alpha)} a (k,l,x)\Big|_{l=k} + R,\] 
   giving the terms in the sum in \eqref{EQ:expamp}.
   % % % % % % % % % % % % % % % % % % % % % % % % %

  Let us now analyse the remainder $R$. It is the sum of terms of the form
   
   \[R_{j}(k,x) = \int_{\mathbb{T}^n} e^{2\pi i k\cdot y} (e^{2\pi i  y} -1 )^\alpha  b_j(k,y,x)\text{d}y,  \]
   with  $|\alpha| = N$ and some $b_j$ containing the terms which are combination of functions
   \[ D_{x}^{\alpha_0}\mathcal{F}_2   a(k,y,x) \ \ \  \ \text{for } \ |\alpha_0|\leq N,\]
   where $\mathcal{F}_2$ means the Fourier transform with respect to the second variable, multiplied by some smooth functions, using the expressions \eqref{EQ:hjs}.
    It follows that for any $\beta$ the functions $D_{x}^{(\beta)} R_j(k,x)$ are the sums of terms of the form 
    \begin{multline*}
    \int_{\mathbb{T}^n} e^{2\pi i k\cdot y} a_j(y) (e^{2\pi i  y} -1 )^\alpha  D_{x}^{(\beta)}D_{x}^{(\alpha_0)}\mathcal{F}_2 a(k,y,x) \text{d}y \\
    =\Delta_{a_j}\Delta_{l}^{\alpha} D_{x}^{(\alpha_0 + \beta)}a(k,l,x) \Big|_{l=k},
    \end{multline*}
   for some smooth functions $a_j\in C^\infty(\Tn)$. Since 
   $\displaystyle a \in\mathcal{A}^{\mu_1,\mu_2}_{\rho,\delta} $,
   by Lemma \ref{lemma1} we obtain that $\displaystyle  R_j(k,x)$ satisfies
   the estimate
   \[|R_j(k,x)| \leq C(1 +|k|)^{\mu_1  } (1 +|k|)^{\mu_2 -\rho|\alpha| + \delta|\alpha_0| + \delta|\beta|},\] 
   for any value of $J$ in \eqref{EQ:amps}.
       Using that $|\alpha| = N$ and $|\alpha_0|\leq N$
  we get that  
   \[|R_j(k,x)| \leq C(1 +|k|)^{\mu_1+\mu_2 -(\rho -\delta)N + \delta|\beta|  }.\]
   Also, for the terms $\Delta^{\beta}_{k} R_j(k,x)$ we can express them as  sums of terms of the form
 $$
   \int_{\mathbb{T}^n} e^{2\pi i k\cdot z}  (e^{2\pi i  z} -1 )^\beta a_j(z)(e^{2\pi i z} - 1)^{\alpha} b_j(k,z,x) \text{d}z,
 $$
   with similar $b_j$ and $a_j$ as above. An argument similar to the one above shows the estimate
   \[|\Delta^{\beta}_{k} R_j(k,x)| \leq C(1 +|k|)^{\mu_1 +\mu_2 -\rho |\beta|  -(\rho-\delta)N }.\]
 By choosing $N$ large enough, arguments like in the classical pseudo-differential calculus imply that we have \eqref{EQ:expamp} .
 \end{proof}

   \section{Symbolic calculus}
\label{SEC:calculus}
   
         In this section we develop elements of the symbolic calculus of pseudo-differential operators on $\Zn$ by deriving formulae for the composition, adjoint, transpose, as well as for the parametrix for elliptic operators.
         
     \begin{thm}[Composition formula]\label{THM:comp}
  Let $0\leq \delta<\rho\leq 1.$ Let $\sigma\in S^{\mu_1}_{\rho, \delta}(\mathbb{Z}^n\times\mathbb T^n)$ and $\tau\in S^{\mu_2}_{\rho, \delta}(\mathbb{Z}^n\times\mathbb T^n)$. Then the composition $\rm{Op}(\sigma)\circ \rm{Op}(\tau)$ is a pseudo-differential operator with symbol $\varsigma\in S^{\mu_1 +\mu_2}_{\rho, \delta}(\mathbb{Z}^n\times\mathbb T^n)$, which can be given as an asymptotic sum
  \begin{equation}\label{EQ:comp}
  \varsigma(k, x)\sim  \sum_{\alpha} \frac{1}{\alpha !} D_{x}^{(\alpha)} \sigma(k, x)\Delta_{k}^{\alpha}\tau(k,x).
  \end{equation}
  \end{thm}
   
   We note that the order of taking differences and derivatives in \eqref{EQ:comp} changes in comparison to the analogous composition formulae on $\Rn$ and $\Tn$, see \cite{RT-JFAA,ruzhansky2009pseudo}.
   
       \begin{proof}[Proof of Theorem \ref{THM:comp}] 
   Let $f ,g \in \mathcal{S}(\mathbb{Z}^n)$. The pseudo-differential operators with symbols $\sigma $ and $\tau$ are given by
   \begin{equation}
   \label{9}
   \text{Op}(\sigma)f(k) = \sum_{m \in \mathbb{Z}^n}\int_{\mathbb{T}^n}e^{2\pi i(k-m)\cdot x}\sigma(k,x)f(m)\text{d}x,
   \end{equation}
   \begin{equation}
   \label{10}
   \text{Op}(\tau)g(m) = \sum_{l \in \mathbb{Z}^n}\int_{\mathbb{T}^n}e^{2\pi i(m-l)\cdot y}\tau(m,y) g(l)\text{d}y.
   \end{equation}
   The composition of \eqref{9} and \eqref{10} gives
   
   \begin{eqnarray*}
   \text{Op}(\sigma)\big(\text{Op}(\tau)g\big)(k) &=&  \sum_{l \in \mathbb{Z}^n}\sum_{m \in \mathbb{Z}^n}\int_{\mathbb{T}^n}\int_{\mathbb{T}^n}e^{2\pi i(k-m)\cdot x}\sigma(k,x) e^{2\pi i(m-l)\cdot y}\tau(m,y) g(l)\text{d}y\text{d}x\\
   &=& \sum_{l \in \mathbb{Z}^n} \int_{\mathbb{T}^n}e^{2\pi i(k-l)\cdot y}\varsigma(k,y)g(l) \text{d}y,
   \end{eqnarray*}
   where 
   \begin{eqnarray*}
   \varsigma(k,y) &=&\sum_{m \in \mathbb{Z}^n}\int_{\mathbb{T}^n} e^{2\pi i(k-m)\cdot(x- y)}\sigma(k,x) \tau(m,y)\text{d}x \\
   &=&\sum_{m \in \mathbb{Z}^n}\int_{\mathbb{T}^n}e^{2\pi ik\cdot(x- y)}e^{-2\pi im\cdot(x- y)}\sigma(k,x) \tau(m,y) \text{d}x \\
   &=& \int_{\mathbb{T}^n}e^{2\pi ik\cdot(x- y)}\sigma(k,x) \widehat{\tau}(x- y,y)\text{d}x  \\
   &=& \int_{\mathbb{T}^n} e^{2\pi ik\cdot x}\sigma(k,y + x) \widehat{\tau}(x,y) \text{d}x,
   \end{eqnarray*}
    where $\widehat{\tau}$ is the Fourier transform of $\tau(m, y)$ in the first variable.
 Using  the periodic Taylor expansion \eqref{EQ:Taylor}  we have
  \begin{eqnarray*}
   \varsigma(k,x) &=& \int_{\mathbb{T}^n} e^{2\pi ik\cdot y}\sigma(k,x + y) \widehat{\tau}(y,x) \text{d}y\\
   &=& \int_{\mathbb{T}^n} e^{2\pi ik\cdot y}  \sum_{|\alpha|< N} \frac{1}{\alpha !} (e^{2\pi i y}- 1)^{\alpha} D_{x}^{(\alpha)} \sigma(k, x) \widehat{\tau}(y,x) \text{d}y  + R,\\ 
   &=& \sum_{|\alpha|< N} \frac{1}{\alpha !} D_{x}^{(\alpha)} \sigma(k, x)\Delta_{k}^{\alpha}\tau(k,x) + R,
   \end{eqnarray*} 
  where $R$ is a remainder coming from the Taylor expansion formula. We observe that since the difference operators satisfy the Leibniz rule we have 
  $$D_{x}^{(\alpha)} \sigma(k, x)\Delta_{k}^{\alpha}\tau(k,x)\in S^{\mu_1+\mu_2-(\rho-\delta)|\alpha|}_{\rho,\delta}(\Zn\times\Tn).$$
  The remainder $R$ can be treated along the lines of the remainder treatment in the proof of Theorem \ref{THM:amplitudes}.
   \end{proof}
   
 The adjoint $T^{*}$ of the operator $T$ on $\ell^2(\Zn)$ is defined by
      \begin{equation}\label{adjoint}
     \big( Tf, g\big)_{\ell^2(\mathbb{Z}^n)} = \big( f, T^{*}g \big)_{\ell^2(\mathbb{Z}^n)}.
     \end{equation}
We now derive the asymptotic formula for its symbol.

  \begin{thm}[Adjoint operators]\label{THM:adjoint}
   Let $0\leq \delta<\rho\leq 1.$ Let $\sigma\in S^{\mu}_{\rho, \delta}(\mathbb{Z}^n\times\mathbb T^n)$. Then there exists a symbol $\sigma^*\in S^{\mu}_{\rho, \delta}(\mathbb{Z}^n\times\mathbb T^n)$ such that the adjoint operator $\rm{Op}(\sigma)^*$ is a pseudo-difference operator with symbol $\sigma^*$, i.e. $\rm{Op}(\sigma)^* = \rm{Op}(\sigma^*)$, and we have  
   \begin{equation}\label{EQ:adj}
\sigma^*(k,x)  \sim \sum_\alpha \frac{1}{\alpha !}\Delta_{k}^{\alpha} D_{x}^{(\alpha)}\overline{\sigma(k,x)}. 
\end{equation}   
\end{thm}
  \begin{proof}
  Since
  \[\big( \text{Op}(\sigma)f,g \big)_{\ell ^2(\mathbb{Z}^n)} = \big( f,\text{Op}(\sigma)^*g \big)_{\ell^2(\mathbb{Z}^n)}\]
  and
     \begin{eqnarray*}
   \big( \text{Op}(\sigma)f,g \big)_{\ell ^2(\mathbb{Z}^n)} &=& \sum_{k \in \mathbb{Z}^n}\sum_{l \in \mathbb{Z}^n}\int_{ \mathbb{T}^n} e^{2\pi i (k-l)\cdot y}\sigma(k,y)  f(l) \overline{g(k)} \text{d}y \\
 &=& \sum_{l \in \mathbb{Z}^n} f(l)\Bigg(\overline{ \sum_{k \in \mathbb{Z}^n} \int_{ \mathbb{T}^n}e^{-2\pi i (k-l)\cdot y} \overline{\sigma(k,y)} g(k)} \text{d}y\Bigg),
  \end{eqnarray*}
  we have  
\[ \text{Op}(\sigma)^* g(l) = \sum_{k \in \mathbb{Z}^n} \int_{ \mathbb{T}^n}e^{-2\pi i (k-l)\cdot y} \overline{\sigma(k,y)} g(k) \text{d}y.\]
  Swapping $k$ and $l$, we can write it as
  \[ \text{Op}(\sigma)^* g(k) = \sum_{k \in \mathbb{Z}^n} \int_{ \mathbb{T}^n}e^{2\pi i (k-l)\cdot y} \overline{\sigma(l,y)} g(l) \text{d}y,\]
  which is an amplitude operator with amplitude $a(k,l,y) = \overline{\sigma(l,y)}\in\mathcal{A}^{0,\mu}_{\rho,\delta}(\Zn\times\Zn\times\Tn)$.
  Then by Theorem \ref{THM:amplitudes} we have $\text{Op}(\sigma)^* = \text {Op}(\sigma^*)$ with
  \[\sigma^*(k,x)  \sim \sum_{\alpha} \frac{1}{\alpha !}\Delta_{l}^{\alpha} D_{x}^{(\alpha)}\overline{\sigma(l,x)}\Big|_{l=k},\]
  yielding \eqref{EQ:adj} and completing the proof.
  \end{proof}
  
           For $f, g\in \mathcal{S}(\mathbb{Z}^n)$ we recall that the transpose $T^t$ of a linear operator $T$ is given by the distributional duality
         \[\langle T^t f,g\rangle =\langle f,T g\rangle,\] 
    which means that for all $k\in \mathbb{Z}^n$ we have
      \[\sum_{k\in \mathbb{Z}^n}(T^t f)(k)g(k) = \sum_{k\in \mathbb{Z}^n}f(k)(Tg )(k). \]
   
   \begin{thm}[Transpose operators]
    Let $0\leq \delta<\rho\leq 1.$ Let $\sigma\in S^{\mu}_{\rho, \delta}(\mathbb{Z}^n\times\mathbb T^n)$. Then there exists a symbol $\sigma^t\in S^{\mu}_{\rho, \delta}(\mathbb{Z}^n\times\mathbb T^n)$ such that the transpose operator $\rm{Op}(\sigma)^t$ is a pseudo-differential operator with symbol $\sigma^t$, i.e. $\rm{Op}(\sigma)^t = \rm{Op}(\sigma^t)$, and we have  
      \begin{equation}\label{EQ:transpose}
   {\sigma^t}(k, x) \sim \sum_{\alpha}\frac{1}{\alpha!}\Delta_{k}^{\alpha}D_{x}^{(\alpha)}\sigma(k, -x).
   \end{equation}
   \end{thm}
   
   \begin{proof}
  By the definition of transpose, we have 
   \[\sum_{k \in \mathbb{Z}^n}(T^t f)(k)g(k) = \sum_{k \in \mathbb{Z}^n}f(k)(Tg )(k).\]
   Then we can calculate
   \begin{eqnarray*}
   \sum_{k \in \mathbb{Z}^n}f(k)(Tg )(k)&= & \sum_{k \in \mathbb{Z}^n} \sum_{y\in \mathbb{Z}^n}\int_{\mathbb{T}^n} f(k)e^{2\pi i(k-l)\cdot x}\sigma(k, x)g(l)\text{d}x\\
    &= & \sum_{l \in \mathbb{Z}^n}g(l)\Bigg(\sum_{k\in \mathbb{Z}^n}\int_{\mathbb{T}^n}e^{2\pi i(k-l)\cdot x}\sigma(k, x)f(k)\text{d}x\Bigg).\\
   \end{eqnarray*}
   Therefore  
   \[ T^t g(l) =\sum_{k\in \mathbb{Z}^n}\int_{\mathbb{T}^n}e^{ 2\pi i(k-l)\cdot x}\sigma(k, x)g(k)\text{d}x. \]
  We can rewrite this, changing $x$ to $-x$, as
   \[ T^t g(l) =\sum_{k\in \mathbb{Z}^n}\int_{\mathbb{T}^n}e^{2\pi i(l-k)\cdot x}\sigma(k, -x)g(k)\text{d}x. \] 
  This is an amplitude operator with amplitude $a(l, k, x) =\sigma(k, -x)$.
   Since  $\sigma \in S^{\mu}_{\rho,\delta} (\mathbb{Z}^n \times \mathbb{T}^n) $, we have $a \in \mathcal{A}^{0,\mu}_{\rho,\delta} (\mathbb{Z}^n \times \mathbb{T}^n) $.
   Applying Theorem \ref{THM:amplitudes} we get that $\displaystyle T^t = \text{Op}(a^t)$ with  $\sigma^t \in S^{\mu}_{\rho,\delta} (\mathbb{Z}^n \times \mathbb{T}^n) $ given by 
   \begin{equation*}
   {\sigma^t}(k, x) \sim \sum_{\alpha}\frac{1}{\alpha!}\Delta_{k}^{\alpha}D_{x}^{(\alpha)}\sigma(k, -x),
   \end{equation*}
   yielding \eqref{EQ:transpose} and completing the proof.
   \end{proof}

   We now record the useful statement on the asymptotic sum of symbols.
   
   \begin{lemma}[Asymptotic sums of symbols]
   Let $1\geq \rho>\delta \geq 0$. Let $\big\{\mu_j\big\}_{j=0}^{\infty}$ be a sequence of $\mu_j\in \mathbb{R}$ such that $\mu_{j}>\mu_{j+1}$ and $\mu_j\rightarrow -\infty$ as $j\rightarrow \infty$.
   Let $\sigma_j\in S_{\rho, \delta}^{\mu_j}(\mathbb{Z}^n\times\mathbb{T}^n)$ for all $j\in \mathbb{N}_0$. Then there exits $S_{\rho, \delta}^{\mu_0}(\mathbb{Z}^n\times\mathbb{T}^n)$ such that for all $N\in \mathbb{N}_0$ we have
   \[\sigma \sim \sum_{j=0}^{\infty}\sigma_j,\]
   that is,
   \[\sigma - \sum_{j=0}^{N-1}\sigma_j \in S_{\rho, \delta}^{\mu_N}(\mathbb{Z}^n\times\mathbb{T}^n), \ \ \text{for all}\ N\in \mathbb{N} . \]
   \end{lemma}
   This statement immediately follows from \cite[Theorem 4.4.1]{ruzhansky2009pseudo} since the symbol classes are the same modulo swapping the order of the variables.
   
   We now define elliptic operators in the symbol classes $S^{\mu}_{\rho, \delta}(\mathbb{Z}^n\times\mathbb T^n)$.
   
        \begin{defn}[Elliptic operators]
      A symbol  $\sigma\in S^{\mu}_{\rho, \delta}(\mathbb{Z}^n\times\mathbb T^n)$ will be called elliptic (of order $\mu$) if there exist $C>0$ and $M>0$ such that 
      \[|\sigma
      (k,x)|\geq C(1+|k|)^{\mu}\]
      holds for all $x\in \mathbb{T}^n$ and all $k\in \mathbb{Z}^n$ such that $|k|\geq M.$ The corresponding pseudo-difference operator  \text{Op}$(\sigma
      )$ is then also called elliptic.
      \end{defn}

   \begin{thm}[Ellipticity and parametrix]\label{THM:elliptic}
   Let $0\leq \delta<\rho \leq 1.$
   An operator $A\in \Op(S^{\mu}_{\rho, \delta}(\mathbb{Z}^n\times\mathbb{T}^n))$ is elliptic if and only if there exists $B\in \Op(S^{-\mu}_{\rho, \delta}(\mathbb{Z}^n\times\mathbb{T}^n))$ such that 
   \[BA\backsim I\backsim AB\;\textrm{ modulo }\; \Op(S^{-\infty}(\mathbb{Z}^n\times \mathbb T^n)),\]
   where $I$ is the identity operator.
   
   Moreover, let $\displaystyle A\sim \sum_{l=0}^{\infty}A_l$ be an expansion, where $A_l\in \Op(S_{\rho,\delta}^{\mu-(\rho-\delta)l}(\mathbb{Z}^n\times \mathbb{T}^n))$. Then an asymptotic expansion  $\displaystyle B\sim \sum_{j=0}^{\infty}B_j$ with $B_j\in \Op(S_{\rho,\delta}^{-\mu-(\rho-\delta)j}(\mathbb{Z}^n\times \mathbb{T}^n))$
    can be obtained by setting $\displaystyle \sigma_{B_0} := \frac{1}{\sigma_{A_0}}$, and then recursively    
   \begin{equation}\label{EQ:parexp}
 \sigma_{B_N}(k, x) = \frac{-1}{\sigma_{A_0}(k, x)} \sum_{j = 0}^{N-1}\sum_{l = 0}^{N-1}\sum_{|\gamma| = N - j - l}\frac{1}{\gamma!}\Big[D_{x}^{(\gamma)} \sigma_{B_j}(k, x)\Big] \Delta_{k}^{\gamma}\sigma_{A_l}(k, x).
\end{equation}
   \end{thm}
   Here, for example, by $AB\backsim I$ we mean that 
   \[I - AB \in S^{-\infty}(\mathbb{Z}^n\times\mathbb{T}^n):=\bigcap_{\nu\in\R} S^{\nu}_{\rho, \delta}(\mathbb{Z}^n\times\mathbb{T}^n),\]
   is the class of `smoothing' operators, independent of $\rho$ and $\delta$. We will still call such operators `smoothing' although the smoothness does not make sense on the lattice $\Zn$.
   
   \begin{proof}
   ``If part". We want to show that $A\in $\ \rm{Op}$(S^{\mu}_{\rho, \delta}(\mathbb{Z}^n\times\mathbb{T}^n))$ is elliptic and we are given that
   \[I - AB = T\in S^{-\infty}(\mathbb{Z}^n\times\mathbb{T}^n).\]
   Hence by Theorem \ref{THM:comp} we have, in particular, that
   \[ 1 - \sigma_{A}(k, x)\sigma_{B}(k, x)\in S^{-(\rho-\delta)}_{\rho, \delta}(\mathbb{Z}^n\times\mathbb{T}^n),\]
   implying that there exists a constant $C>0$ such that
   \[ |1 - \sigma_{A}(k, x)\sigma_{B}(k, x)| \leq  C(1+|k|)^{-(\rho-\delta)}.\]
  Take $M$ such that $\displaystyle C(1+|M|)^{-(\rho - \delta)}< \frac{1}{2}.$ 
   It then follows that 
   \[|\sigma_{A}(k, x)\sigma_{B}(k, x)| \geq \frac{1}{2}, \quad \text{for all} \ \ |k|\geq M,\]
   and hence
   \[|\sigma_{A}(k, x)| \geq \frac{1}{2|\sigma_{B}(k, x)|} \geq \frac{1}{2C_B}(1+|k|)^{\mu},\]
   since 
   \[|\sigma_{B}(k, x)|\leq  C_B(1+|k|)^{-\mu}.\] Hence $\sigma_A$
   is elliptic of order $\mu$.

   ``Only if part". We can restrict to $|k|\geq M$. Take 
   \[\sigma_{B_0}(k, x):= \frac{1}{\sigma_{A}(k, x)}.\]
   By \cite[Lemma 4.9.4]{ruzhansky2009pseudo} we have $\sigma_{B_0}\in S_{\rho, \delta}^{-\mu}(\mathbb{Z}^n\times \mathbb{T}^n)$. Also by the composition formula in Theorem \ref{THM:comp} we have
   \[\sigma_{B_0 A} = \sigma_{B_0}\sigma_{A} - \sigma_{T} \backsim 1 - \sigma_{T},\]
   for some $T\in S_{\rho, \delta}^{-(\rho-\delta)}(\mathbb{Z}^n\times \mathbb{T}^n)$, hence $B_0 A = I -T.$   
   The rest of the ``only if'' proof follows by the composition formula and a functional analytic argument similar to the proof of \cite[Theorem 4.9.6]{ruzhansky2009pseudo}, so we omit it.
   
   We will now show formula \eqref{EQ:parexp}.
 We have $I \backsim BA$ which means that $1  \backsim  \sigma_{BA}(k, x).$ 
   Then by Theorem \ref{THM:comp} we have
   \begin{eqnarray*}
   1 &\backsim & \sum_{\gamma \geq 0}\frac{1}{\gamma !}\Big[D_{x}^{(\gamma)}\sigma_B(k, x)\Big]\Delta_{k}^{\gamma} \sigma_A(k, x) \\
   &\backsim & \sum_{\gamma \geq 0}\frac{1}{\gamma !}\Big[D_{x}^{(\gamma)} \sum_{j =0}^{\infty}\sigma_{B_j} (k, x)\Big]\Delta_{k}^{\gamma} \sum_{l = 0}^{\infty}\sigma_{A_l}(k, x).
   \end{eqnarray*}
   The rest follows by using a similar argument to the proof of \cite[Theorem 4.9.13]{ruzhansky2009pseudo}, completing the proof. 
   \end{proof}
   
   \section{Relation between lattice and toroidal quantizations}
   \label{SEC:link}
   
   We will now discuss the relation between the lattice quantization analysed so far and the toroidal quantization developed in \cite{RT-JFAA,RT-Birk}. The toroidal quantization has since led to many further developments and applications, see e.g. \cite{Paycha16,PZ14,ParZ14,Cardona14}, to mention a few. So, the described link leads to a way of transferring results from the toroidal setting to the lattice. We will give such an example in the derivation of $\ell^2$-estimates in Theorem \ref{THM:L2}, as well as apply it in Corollary \ref{COR:comp} to give a characterisation of compact operators on $\ell^2(\Zn)$, in Corollary \ref{COR:compG} for a version of the Gohberg lemma,and in Theorem \ref{COR:nuc} to give a condition for the membership in the Schatten-von Neumann classes.
   
   To distinguish between these two quantizations, here we will use the notation $\Optn$ for the toroidal quantization with symbol $\tau:\Tn\times\Zn\to\C$, for $v\in C^\infty(\Tn)$ yielding
   \begin{equation}\label{EQ:qtor}
 \Optn(\tau)v(x)=\sum_{k\in\Zn} e^{2\pi i x\cdot k}\tau(x,k) (\Ftn v)(k).
\end{equation} 
   To contrast it with the lattice quantization \eqref{neil}, we will denote it here by   
    \begin{equation}\label{neiltn}
                    \Opzn(\sigma)f(k) := \int_{\mathbb{T}^n}e^{2\pi i k\cdot x}\sigma(k,x)(\Fzn f)(x)\text{d}x, \ \ \ k\in \mathbb{Z}^n.
                    \end{equation}
The lattice Fourier transform $\Fzn$ in \eqref{gael1} is related to the toroidal Fourier transform by
\begin{equation}\label{EQ:ftts}
 \Ftn v(k)=\int_\Tn e^{-2\pi i x\cdot k} v(x)\text{d}x=\Fzn^{-1}v(-k),
\end{equation} 
since $\Fzn^{-1}$ has the form \eqref{EQ:Finv}.

 \begin{thm}\label{THM:link}
For a function $\sigma:\Zn\times\Tn\to\C$ define $\tau(x,k):=\overline{\sigma(-k,x)}$. Then we have
\begin{equation}\label{EQ:link1}
\Opzn(\sigma)=\Fzn^{-1}\circ \Optn(\tau)^*\circ \Fzn,
\end{equation} 
where $\Optn(\tau)^*$ is the adjoint of the toroidal pseudo-differential operator $\Optn(\tau).$
We also have
\begin{equation}\label{EQ:link2}
\Optn(\tau)=\Fzn\circ \Opzn(\sigma)^*\circ \Fzn^{-1},
\end{equation} 
where $\Opzn(\sigma)^*$ is the adjoint of the pseudo-difference operator $\Opzn(\sigma).$
\end{thm} 

Formulae \eqref{EQ:link1} and \eqref{EQ:link2} allow one to reduce certain problems for $\Opzn$ to the corresponding problems for $\Optn$, at least when one is working in the $\ell^2$-framework. In Theorem \ref{THM:L2} we show this in the case of finding conditions for $\Opzn(\sigma)$ to be bounded on $\ell^2(\Zn)$ in terms of $\sigma$. 
Moreover, we can conclude that $\Opzn(\sigma)$ is in the $p$-Schatten class on $\ell^2(\Zn)$ if the operator $\Optn(\tau)$ is in the $p$-Schatten class on $L^2(\Tn)$, and conditions for toroidal pseudo-differential operators to be in the $p$-Schatten classes or to be $r$-nuclear on $L^2(\Tn)$ in terms of their toroidal symbols were given in \cite{DR-MRL} and also in \cite{DR-JMPA}. We will give such an application in Theorem \ref{COR:nuc}.

\begin{proof}[Proof of Theorem \ref{THM:link}]
For $g\in C^\infty(\Tn$), consider the operator
$$
  Tg(k) := \int_{\mathbb{T}^n}e^{2\pi i k\cdot x}\sigma(k,x)g(x)\text{d}x.
$$
Then by \eqref{neiltn} we have the relation
\begin{equation}\label{EQ:rel1}
\Opzn(\sigma) = T\circ\Fzn.
\end{equation} 
Let us calculate the adjoint operator $T^*$ determined by the relation
$$
(Tg,h)_{\ell^2(\Zn)}=(g,T^*h)_{L^2(\Tn)}.
$$
We have
\begin{multline*}
(Tg,h)_{\ell^2(\Zn)}=\sum_{k\in\Zn} Tg(k)\overline{h(k)}=
\sum_{k\in\Zn} \int_{\mathbb{T}^n}e^{2\pi i k\cdot x}\sigma(k,x)g(x)\overline{h(k)} \text{d}x 
\\
= \int_{\mathbb{T}^n} g(x)\left(\sum_{k\in\Zn}e^{2\pi i k\cdot x}\sigma(k,x)\overline{h(k)}\right) \text{d}x.
\end{multline*} 
Consequently, we have
\begin{equation}\label{EQ:Tadj}
\begin{aligned}
T^*h(x) & =\sum_{k\in\Zn}e^{-2\pi i k\cdot x}\overline{\sigma(k,x)}h(k) \\
& = \sum_{k\in\Zn}e^{2\pi i k\cdot x}\overline{\sigma(-k,x)}h(-k) \\
& = \sum_{k\in\Zn}e^{2\pi i k\cdot x}\tau(x,k) \Ftn v(k)=\Optn(\tau)v(x),
\end{aligned}
\end{equation} 
with $v$ such that $\Ftn v(k)=h(-k).$ It then follows from \eqref{EQ:ftts} that 
$h(k)=\Ftn v(-k)=\Fzn^{-1}v(k)$. This and \eqref{EQ:Tadj} imply that
\begin{equation}\label{EQ:rel2}
T^*=\Optn(\tau) \circ\Fzn.
\end{equation}
Consequently, we also have
\begin{equation}\label{EQ:rel3}
T=\Fzn^{*}\circ \Optn(\tau)^* = \Fzn^{-1}\circ \Optn(\tau)^*,
\end{equation}
in view of the unitarity of all the Fourier transforms. And now, combining \eqref{EQ:rel3} with \eqref{EQ:rel1}, we obtain \eqref{EQ:link1}. Finally, \eqref{EQ:link2} follows form \eqref{EQ:link1} by taking adjoint and using the unitarity of the Fourier transform.
\end{proof}

     \section{Applications}
   \label{SEC:L2-HS}
   
   In this section we give conditions for the boundedness of pseudo-difference operators on $\ell^2(\Zn)$, weighted $\ell^2(\Zn)$, $\ell^p(\Zn)$. We also discuss a condition for Hilbert-Schmidt operators and its implication for the $\ell^p$-$\ell^{p'}$ boundedness, and give conditions for the membership in Schatten classes. Finally, we discuss a version of Fourier integral operators on the lattice $\Zn$.
   
   \subsection{Boundedness on $\ell^2(\Zn)$}
   
   We recall that if $H$ be a complex separable Hilbert space then a bounded linear operator on $H$ is said to be a Hilbert-Schmidt operator if there exists an orthonormal basis $\displaystyle \{w_{m}\}_{m=1}^{\infty}$ in $H$ such that
   $\sum_{m=1}^{\infty}\|Aw_{m}\|_{H}^{2}<\infty.$
   If $A\in \mathscr{L}(H)$ is a Hilbert-Schmidt operator then its norm is given by
   \[\|A\|_{\HS}^{2} = \sum_{m=1}^{\infty}\|Aw_{m}\|_{H}^{2},\]
   where $\displaystyle \{w_{m}\}_{m=1}^{\infty}$ is any orthonormal basis in $H$.
   The following is a natural condition for an operator on $\ell^2(\Zn)$ to be Hilbert-Schmidt in terms of the symbol. Interestingly, it implies that Hilbert-Schmidt operators are $\ell^p$-$\ell^{p'}$ bounded for all $1\leq p\leq 2$.
     
   \begin{proposition}\label{PROP:HS}
   The pseudo-difference operator $\Op(\sigma): \ell^{2}({\mathbb Z}^n)\rightarrow \ell ^{2}({ \mathbb Z}^n)$ is a Hilbert-Schmidt operator if and only if $\sigma\in L^{2}(\mathbb{Z}^{n}\times\mathbb{T}^n)$, in which case we have
   \begin{equation}\label{EQ:HS}
\|\Op(\sigma)\|_{\HS} = \|\sigma\|_{L^{2}(\mathbb{Z}^{n}\times\mathbb{T}^n)}
   = \left(\sum_{k\in \mathbb Z^n}\int_{\mathbb{T}^n}|\sigma(k,x)|^{2}\text{\rm d}x\right)^{\frac12}.
\end{equation} 
   Moreover,  if $\sigma\in L^2(\Zn\times\Tn)$ then $\Op(\sigma):\ell^p(\Zn)\to\ell^q(\Zn)$ is bounded for all $1\leq p\leq 2$ and $\frac1p+\frac1q=1$, and we have
   \begin{equation}\label{EQ:lplq}
\|\Op(\sigma)\|_{\mathscr{L}(\ell^p(\Zn)\to \ell^q(\Zn))}\leq \|\sigma\|_{L^{2}(\mathbb{Z}^{n}\times\mathbb{T}^n)}.
\end{equation} 
   \end{proposition}
   
   The Hilbert-Schmidt part follows directly by the Plancherel formula and will be used in Theorem \ref{COR:nuc} to imply Schatten properties of operators. The boundedness part was shown in \cite{CR2011}.
   We will give simple proofs for completeness, to show formulae \eqref{EQ:HS} and \eqref{EQ:lplq}.
   In the case $p=2$ it implies the $\ell^2$-boundedness result in \cite{molahajloo2009pseudo} (where $n=1$ was considered).

   \begin{proof}[Proof of Proposition \ref{PROP:HS}]
   Let $\{w_{m}\}_{m\in\mathbb{Z}^n}$ be the standard orthonormal basis  for $\ell^{2}(\mathbb{Z}^n)$ which is defined by $w_{m}(k)=\delta_{mk}$ being the Kronecker delta.
      By \eqref{gael} the Fourier transform of $w_{m}$ is given by 
   $\widehat{w_{m}}(y) = e^{-2\pi im\cdot y}.$
   Then we have
  \begin{equation}\label{EQ:www}
   \text{Op}(\sigma)w_{m}(k) = \int_{\mathbb{T}^n}e^{2\pi ik\cdot y}\sigma(k,y)e^{-2\pi im\cdot y}dy = (\mathcal{F}_{\mathbb{T}^n}\sigma)(k,m-k).
\end{equation} 
   We then have
   \begin{multline*}
    \|\Op(\sigma)\|_{\HS}^{2} = \sum_{m\in \mathbb Z^n}\|\Op(\sigma)w_{m}\|_{\ell^{2}(\mathbb{Z}^n)}^{2}= \sum_{m\in \mathbb Z^n} \sum_{k\in \mathbb Z^n}|(\mathcal{F}_{\mathbb{T}^n}\sigma)(k,m-k)|^{2}\\
   =  \sum_{k\in \mathbb Z^n}\sum_{m\in \mathbb Z^n}|(\mathcal{F}_{\mathbb{T}^n}\sigma)(k,m)|^{2} 
   = \sum_{k\in \mathbb Z^n}\int_{\mathbb{T}^n}|\sigma(k,y)|^{2}\text{d}y
   = \|\sigma\|_{L^{2}(\mathbb{Z}^{n}\times\mathbb{T}^n)}^{2},
\end{multline*} 
   completing the proof of the Hilbert-Schmidt part.
   
  Furthermore, under the condition $\sigma\in L^{2}(\mathbb{Z}^{n}\times\mathbb{T}^n)$ we have 
  the $\ell^1-\ell^\infty$ boundedness in view of 
\begin{multline*}
|\Op(\sigma)f(k)|\leq \int_\Tn |\sigma(k,x)| |\widehat{f}(x)|\text{d}x
\leq \|\widehat{f}\|_{L^\infty(\Tn)}
\int_\Tn |\sigma(k,x)| \text{d}x
\\
\leq  \|f\|_{\ell^1(\Zn)}
\left(\int_\Tn |\sigma(k,x)|^2  \text{d}x\right)^{\frac12}
\leq \|f\|_{\ell^1(\Zn)}\|\sigma\|_{L^{2}(\mathbb{Z}^{n}\times\mathbb{T}^n)}.
\end{multline*} 
We also have the $\ell^2-\ell^2$ boundedness if we apply the Cauchy-Schwartz inequality to the first line above:
$$
|\Op(\sigma)f(k)|^2\leq \left(\int_\Tn |\sigma(k,x)|^2\text{d}x\right) \left(\int_\Tn |\widehat{f}(x)|^2\text{d}x\right),
$$
so that the result follows by interpolation.
   \end{proof}
   
   We now improve the statement of Proposition \ref{PROP:HS} in the case of $p=2$ showing that actually no decay is needed for the $\ell^2$-boundedness provided that finitely many derivatives are bounded, yielding a Mikhlin type theorem, but for general pseudo-differential operators.
   
   \begin{thm}\label{THM:L2}
 Let $\varkappa\in\mathbb N$ and $\varkappa>n/2$.
Assume that the symbol $\sigma:\Zn\times\Tn\to\C$ satisfies
\begin{equation}\label{EQ:l2conds}
|\partial_x^\beta \sigma(k,x)|\leq C, \quad\textrm{ for all }\; (k,x)\in\Zn\times\Tn,
\end{equation} 
for all $|\beta|\leq \varkappa$. Then $\Op(\sigma)$ extends to a bounded operator on $\ell^2(\Zn)$.
\end{thm} 

\begin{proof}
Using the equality \eqref{EQ:link1} in Theorem \ref{THM:link} and the fact that the Fourier transform $\Fzn$ is an isometry from $\ell^2(\Zn)$ to $L^2(\Tn)$, 
it follows that $\Op(\sigma)\equiv \Opzn(\sigma)$ is bounded on $\ell^2(\Zn)$ if and only if
$\Optn(\tau)$ is bounded on $L^2(\Tn)$ for the toroidal symbol
$\tau(x,k)=\overline{\sigma(-k,x)}$. But $\Optn(\tau)$ is bounded on $L^2(\Tn)$ under conditions \eqref{EQ:l2conds} in view of \cite[Theorem 4.8.1]{ruzhansky2009pseudo}.
\end{proof}

\subsection{Compactness, Gohberg lemma, and Schatten-von Neumann classes}

In this section we give applications of the developed calculus to presenting conditions ensuring that the corresponding operators belong to Schatten classes. As usual, an operator is in the $p$-Schatten class if it is compact and if the sequence of its singular numbers is in $\ell^p$.

We start by giving a criterium for compactness of pseudo-differential operators acting on $\ell^2(\Zn)$ and an estimate for their essential spectrum. We recall that for a closed linear operator $A$ on a complex Hilbert space, its essential spectrum is defined as
\begin{equation}\label{EQ:ess}
\Sigma_{ess}(A)={\mathbb C}\backslash\{\lambda\in\mathbb{C}: A-\lambda I \textrm{ is Fredholm and its index is } 0\}. 
\end{equation} 
Then, since the compactness, Fredhomlness, and the index are preserved by unitary operators,
as a consequence of \eqref{EQ:link1} in Theorem \ref{THM:link} and \cite[Theorem 3.2]{DR:Gohberg} we obtain

\begin{cor}\label{COR:comp}
Let $\sigma\in S^0(\Zn\times\Tn)$. Define
\begin{equation}\label{EQ:dd}
d:=\limsup_{|k|\to\infty} \sup_{x\in\Tn} |\sigma(k,x)|.
\end{equation} 
Then
$$
\Opzn(\sigma)  \textrm{ is compact on } \ell^2(\Zn) \textrm{ if and only if }  d=0.
$$
Moreover, we have
$$
\Sigma_{ess}(\Opzn(\sigma)) \subset \{\lambda\in\mathbb{C}:|\lambda|\leq d\}.
$$
\end{cor} 

We also have the following estimates for the distance between a given operator and the space of compact operators on $\ell^2(\Zn)$. In view of \cite{Gohberg}, such a statement is often called Gohberg lemma in the literature, and see e.g. \cite{Mol11,Pir11} also for the results on the circle $\mathbb T^1$, and \cite{DR:Gohberg} for the Gohberg lemma on general compact Lie groups. 
For the following statement  we may recall the notions of generalised difference operators $\Delta_q$ from \eqref{deltaoperator12}.
As a consequence of \eqref{EQ:link1} in Theorem \ref{THM:link} and \cite[Theorem 3.1]{DR:Gohberg} we have

\begin{cor}\label{COR:compG}
Let $\sigma:\Zn\times\Tn\to\C$ be such that
\begin{equation}\label{EQ:Gcond}
|\sigma(k,x)|\leq C,\quad |\nabla_x\sigma(k,x)|\leq C,\quad |\Delta_q\sigma(k,x)|\leq C(1+|k|)^{-\rho},
\end{equation} 
for some $\rho>0$, for all $q\in C^\infty(\Tn)$ with $q(0)=0$ and all $(k,x)\in\Zn\times\Tn$.
Then for all compact operators $K$ on $\ell^2(\Zn)$ we have
$$
\|\Opzn(\sigma)-K\|_{\mathcal{L}(\ell^2(\Zn))}\geq \limsup_{|k|\to\infty} \sup_{x\in\Tn} |\sigma(k,x)|.
$$
In particular, this conclusion holds for any $\sigma\in S^0(\Zn\times\Tn).$
\end{cor}

The starting point for the analysis of the Schatten classes is the following condition ensuring the membership in $p$-Schatten classes for $2\leq p< \infty$. Thus for $2\leq p< \infty$ and $\frac1p+\frac1{p'}=1$ we have
 \begin{equation}\label{EQ:schpp}
\sum_{k\in\Zn} \|\sigma(k,\cdot)\|_{L^{p'}(\Tn)}^{p'}<\infty
\;\Longrightarrow \;
\Opzn(\sigma)\textrm{ is } p\textrm{-Schatten operator on } \ell^2(\Zn).
\end{equation}
In fact, \eqref{EQ:schpp} holds in much greater generality, in particular, on all locally compact separable unimodular groups of Type I, see \cite[Corollary 3.18]{Mantoiu-Ruzhansky-DM}. Essentially, it follows by complex interpolation between the Hilbert-Schmidt condition in Proposition \ref{PROP:HS} and the fact that operators with symbols satisfying
$\sum_{k\in\Zn} \|\sigma(k,\cdot)\|_{L^1(\Tn)}<\infty$ are bounded on $\ell^2(\Zn)$.

For $0< p\leq 2$, the membership of operators in $p$-Schatten classes is more difficult to ensure.
However, as a consequence of \eqref{EQ:link1} in Theorem \ref{THM:link} and \cite[Corollary 3.12]{DR-JMPA} we obtain the following statement.

\begin{thm}\label{COR:nuc}
 Let $0< p\leq 2$. Then we have
 \begin{equation}\label{EQ:sch}
\sum_{k\in\Zn} \|\sigma(k,\cdot)\|_{L^2(\Tn)}^p<\infty
\;\Longrightarrow \;
\Opzn(\sigma)\textrm{ is } p\textrm{-Schatten operator on } \ell^2(\Zn).
\end{equation}
In particular, if 
\begin{equation}\label{EQ:trcl}
\sum_{k\in\Zn} \|\sigma(k,\cdot)\|_{L^2(\Tn)}<\infty,
\end{equation} 
 then $\Opzn(\sigma)$ is a trace class operator on $\ell^2(\Zn)$, and in this case we have
 \begin{equation}\label{EQ:trace}
{\rm Tr}\,\p{ \Opzn(\sigma)} = \sum_{k\in\Zn} \int_{\mathbb{T}^n}\sigma(k, x)\text{d}x=\sum_j\lambda_j,
\end{equation}  
where $\{\lambda_j\}_j$ are the eigenvalues of $\Opzn(\sigma)$ counted with multiplicities.

\end{thm} 

\begin{proof}[Proof of Theorem \ref{COR:nuc}]
The conclusion \eqref{EQ:sch} for $0<p\leq 1$ is an immediate consequence of \eqref{EQ:link1} and \cite[Corollary 3.12]{DR-JMPA}, which shows the $p$-nuclearity of $\Opzn(\sigma)$ on $\ell^2(\Zn)$.
Since the notions of $p$-nuclearity and $p$-Schatten classes coincide for Hilbert space (see Oloff \cite{Oloff72}
or Pietsch \cite[Section 6.3.2.11]{Pietsch-history}) we get \eqref{EQ:sch} for $0<p\leq 1$.

As a special case with $p=1$, the operators satisfying \eqref{EQ:trcl} are trace class. 
The first equality in \eqref{EQ:trace} follows from the expression for the kernel at the diagonal given 
in \eqref{EQ:kerdiag}, and the second equality in \eqref{EQ:trace} is the famous Lidskii formula \cite{Lidskii59}.

Consequently, \eqref{EQ:sch} for $1\leq p\leq 2$ follows by interpolation between \eqref{EQ:trcl} and the Hilbert-Schmidt condition \eqref{EQ:HS}. 
\end{proof}

The notion of $r$-nuclearity was introduced and developed by Grothedieck in \cite{Groth-MAMS}.
We can refer e.g. to \cite{DR-JMPA} for the discussion of $r$-nuclearity and its meaning and consequences, and to \cite{Pietsch-history} for an extensive presentation and the history. The direct $r$-nuclearity considerations in our setting appear to be more difficult than those when the space is compact (\cite{DRT-JMPA}) because the kernel does not allow for a natural discrete tensor product decomposition.

\subsection{Weighted $\ell^2$-boundedness}

  For $s\in \R$ and $1\leq p<\infty$ let us define the weighted space $\ell^p_s(\Zn)$ as the space of all $f:\Zn\to\C$ such that 
   \begin{equation}\label{EQ:l2s}
 \|f\|_{\ell^p_s(\Zn)}:=\left(\sum_{k\in\Zn} (1+|k|)^{sp} |f(k)|^p\right)^{1/p}<\infty.
\end{equation} 
   We observe that the symbol $a_s(k)=(1+|k|)^{s}$ belongs to $S^s_{1,0}(\Zn\times\Tn)$, and we have
   $f\in \ell^p_s(\Zn)$ if and only if $\Op(a_s) f\in \ell^p(\Zn)$. Consequently, we have
   \begin{equation}\label{EQ:l2s}
 \ell^p_s(\Zn)=\Op(a_{-s})(\ell^p(\Zn)).
\end{equation} 
    Then 
   Theorem \ref{THM:L2} and Theorem \ref{THM:comp} imply the following boundedness results.
   
     \begin{cor}\label{COR:L2}
   Let $\mu\in\R$ and let $\sigma\in S^\mu_{0,0}(\Zn\times\Tn)$. Then $\Op(\sigma)$ is bounded from  
   $\ell^2_{s}(\Zn)$ to $\ell^2_{s-\mu}(\Zn)$ for all $s\in\R$.
      \end{cor}
      \begin{proof}
  If $A\in \Op(S^\mu_{0,0}(\Zn\times\Tn))$ then by using Theorem \ref{THM:comp} the operator $$B=\Op(a_{s-\mu})A\Op(a_{-s})$$ has symbol in $S^{0}_{0,0}(\Zn\times\Tn)$. Here we observe that we can also include the case $\rho=\delta=0$ in the statement of Theorem \ref{THM:comp} since 
  actually $a_s\in S^s_{1,0}(\Zn\times\Tn)$ so that the asymptotic formulae work also in this case.
  
 Let $f\in \ell^2_{s}(\Zn)$. Then $g:=\Op(a_{s})f\in \ell^2(\Zn)$. By Theorem \ref{THM:L2} the operator $B$ is bounded on $\ell^2(\Zn)$, so that $Bg\in \ell^2(\Zn)$.
 Now, writing
  $$Af=\Op(a_{\mu-s})\Op(a_{s-\mu})A\Op(a_{-s})\Op(a_{s})f=\Op(a_{\mu-s}) B g,$$
  we get that $Af\in \Op(a_{\mu-s}) \ell^2(\Zn)=\ell^2_{s-\mu}(\Zn)$ in view of 
  \eqref{EQ:l2s}.
 Consequently, $A$ is bounded from $\ell^2_{s}(\Zn)$ to $\ell^2_{s-\mu}(\Zn)$.
\end{proof}

 \subsection{G{\aa}rding and sharp G{\aa}rding inequalities on $\mathbb Z^n$}
 
First, let us recall a special case of the G{\aa}rding inequality on compact Lie groups, in the special case of the torus $\mathbb{T}^n$, as in \cite[Corollary 6.2]{Ruzhansky-Wirth:functional-calculus}, stating:  
\begin{cor}\label{A}
Let $0\leq \delta < \rho \leq 1$ and $m > 0$. Let $ B \in \Op _{\mathbb{T}^n}{S}^{2m}_{\rho,\delta}(\mathbb{T}^n \times \mathbb{Z}^n)$ be elliptic  such that $\sigma_B (x,k) \geq  0$, for all $x$ and co-finitely many $k$.
Then there exist  $C_0,C_1 > 0$ such that for all $f \in H^m(\mathbb{T}^n)$ we have
\begin{equation*}
{\rm Re} (Bf,f)_{L^2(\mathbb{T}^n)} \geq C_0||f||^{2}_{H^m(\mathbb{T}^n)}-C_1||f||^{2}_{L^2(\mathbb{T}^n)}.
\end{equation*}
\end{cor}
Let us show that Corollary \ref{A} implies the corresponding G{\aa}rding inequality on $\mathbb{Z}^n$. As there is no regularity concept on the lattice, the statement is given in terms of weighted $\ell^2$-spaces.

\begin{thm}[G{\aa}rding inequality on $\mathbb{Z}^n$]\label{THM:Garding}
Let $0\leq \delta < \rho \leq 1$ and $m > 0$. Let $ A \in \Op _{\mathbb{Z}^n} {S}^{2m}_{\rho,\delta}(\mathbb{Z}^n \times \mathbb{T}^n)$ be elliptic  such that $\sigma_A (k,x) \geq  0$ for all  $x$ and for co-finitely many $k$.
Then there exist  $C_1,C_2 > 0$ such that for all $f \in \ell^{2}_m(\mathbb{T}^n)$  we have
\begin{equation}\label{gard 1}
{\rm Re} (Af,f)_{\ell^2(\mathbb{Z}^n)} \geq C_0||f||^{2}_{\ell^{2}_m(\mathbb{Z}^n)}-C_1||f||^{2}_{\ell^2(\mathbb{Z}^n)}.
\end{equation}
\end{thm}
\begin{proof}
Let $\tau( x,k) =  \overline{\sigma_A(-k,x) } $.
Then we have  
\begin{equation}\label{gard 2}
A = \Op_{\mathbb{Z}^n}(\sigma_A) = \mathcal{F}^{-1}_{\mathbb{Z}^n} \circ \Op _{\mathbb{T}^n}(\tau)^* \circ \mathcal{F}_{\mathbb{Z}^n},
\end{equation} 
by Theorem \ref{THM:link}, and
if  $\sigma_A $ is elliptic on  $\mathbb{T}^n$,  then  $\tau $ is elliptic on  $\mathbb{Z}^n$. Also, if $\sigma_A \geq 0 $, then $\tau \geq 0 $. Then by  Corollary \ref{A}, for all $g \in H^m(\mathbb{T}^n)$ we have 
\begin{multline}\label{gard 3}
{\rm Re}( \Op_{\mathbb{T}^n}(\tau)^* g,g)_{L^2(\mathbb{T}^n)}  \\
={\rm Re}( \Op_{\mathbb{T}^n}(\tau)g,g)_{L^2(\mathbb{T}^n)} \geq   C_0||g||^{2}_{H^{m}(\mathbb{T}^n)}-C_1||g||^{2}_{L^2(\mathbb{T}^n)}.
\end{multline} 
Let $g$ be the Fourier transform of the function $f$, that is $g = \mathcal{F}_{\mathbb{Z}^n}f$.
If $f \in H^m (\mathbb{T}^n)$ then $g \in \ell^{2}_m(\mathbb{Z}^n)$ and 
\begin{equation}\label{gard 4}
\| f\|_{ H^m (\mathbb{T}^n)} = \| f\|_{\ell^{2}_m(\mathbb{Z}^n)} \qquad \text{and} \qquad \| f\|_{ L^2 (\mathbb{T}^n)} = \|f\|_{\ell^{2}(\mathbb{Z}^n)}.
\end{equation} 
By Theorem  \ref{THM:link},
\begin{equation}\label{gard 41}
Af = \mathcal{F}^{-1}_{\mathbb{Z}^n} \circ \Op _{\mathbb{T}^n}(\tau)^* \circ \mathcal{F}_{\mathbb{Z}^n} f =  \mathcal{F}^{-1}_{\mathbb{Z}^n} \circ \Op _{\mathbb{T}^n}(\tau)^*  g,
\end{equation} 
so that $  \mathcal{F}_{\mathbb{Z}^n}Af =   \Op _{\mathbb{T}^n}(\tau)^*  g $. 
Substituting \eqref{gard 4} into \eqref{gard 3} we get 
\begin{eqnarray*}
{\rm Re}( \Op_{\mathbb{T}^n}(\tau)^* g,g)_{L^2(\mathbb{T}^n)} & \geq &   C_0||f||^{2}_{\ell^{2}_m(\mathbb{Z}^n)}-C_1||f||^{2}_{\ell^2(\mathbb{Z}^n)},\\
{\rm Re}( \mathcal{F}_{\mathbb{Z}^n} Af, \mathcal{F}_{\mathbb{Z}^n} f)_{L^2(\mathbb{T}^n)}  & \geq &   C_0||f||^{2}_{\ell^{2}_m(\mathbb{Z}^n)}-C_1||f||^{2}_{\ell^2(\mathbb{Z}^n)},\\
{\rm Re}( \mathcal{F}_{\mathbb{Z}^n}^* \mathcal{F}_{\mathbb{Z}^n} Af,  f)_{\ell^2(\mathbb{Z}^n)}  & \geq &   C_0||f||^{2}_{\ell^{2}_m(\mathbb{Z}^n)}-C_1||f||^{2}_{\ell^2(\mathbb{Z}^n)} \qquad \text {since} \ \mathcal{F}_{\mathbb{Z}^n}^* \mathcal{F}_{\mathbb{Z}^n} = Id, \\
{\rm Re}( Af,f)_{\ell^2(\mathbb{Z}^n)}  & \geq &   C_0||f||^{2}_{\ell^{2}_m(\mathbb{Z}^n)}-C_1||f||^{2}_{\ell^2(\mathbb{Z}^n)}.
\end{eqnarray*}
This completes the proof.
\end{proof}

We now proceed by establishing the sharp G{\aa}rding inequality on $\mathbb{Z}^n$.
Let us recall a special case of the sharp G{\aa}rding inequality on compact Lie groups, in the setting of the torus, see  \cite[Theorem 2.1]{ruzhansky2011sharp}:
\begin{thm}[Sharp G{\aa}rding inequality on $\mathbb{T}^n$]\label{THM:sG}
Let $B \in \Op_{\mathbb{T}^n}{S}^m(\mathbb{T}^n \times \mathbb{Z}^n)$ be such that its symbol ${\sigma}(x,k) \geq 0$ for all $(x,k) \in \mathbb{T}^n \times \mathbb{Z}^n$. Then there exists $C < \infty$ such that we have
\[{\rm Re}( Bg,g)_{L^2(\mathbb{T}^n)} \geq -C\|g\|_{H^{\frac{m-1}{2}}(\mathbb{T}^n)},\]
for all $g \in H^{\frac{m-1}{2}}(\mathbb{T}^n)$.
\end{thm}

Similarly, extending this theorem to thelattice we have the corresponding result:
\begin{thm}[Sharp G\"{a}rding inequality on $\mathbb{Z}^n$]
Let $A \in \Op_{\mathbb{Z}^n}{S}^m(\mathbb{Z}^n \times \mathbb{T}^n)$ be such that its symbol satisfies $\sigma_A(k,x) \geq 0$ for all $(k,x) \in \mathbb{Z}^n \times \mathbb{T}^n$. Then there exists $C < \infty$ such that  we have
\[{\rm Re}( Af,f)_{\ell^2(\mathbb{Z}^n)} \geq -C\|f\|_{\ell^{2}_{\frac{m-1}{2}}(\mathbb{Z}^n)}\]
for all $f \in \ell^{2}_{\frac{m-1}{2}}(\mathbb{Z}^n)$.
\end{thm}
\begin{proof}
Let $\tau( x,k) =  \overline{\sigma_A(-k,x) } $.
Then by Theorem \ref{THM:link} we have  
\begin{equation}\label{gard 2}
A = \Op_{\mathbb{Z}^n}(\sigma_A) = \mathcal{F}^{-1}_{\mathbb{Z}^n} \circ \Op _{\mathbb{T}^n}(\tau)^* \circ \mathcal{F}_{\mathbb{Z}^n}.
\end{equation}
Employing the same argument and  notation as in the proof of Theorem \ref{THM:Garding}, using Theorem \ref{THM:sG}, we get
\begin{eqnarray*}
{\rm Re}(Af,f)_{\ell^2(\mathbb{Z}^n)} &=& Re( \mathcal{F}^{-1}_{\mathbb{Z}^n}  \Op_{\mathbb{T}^n}(\tau)^* g, \mathcal{F}^{-1}_{\mathbb{Z}^n} g)_{\ell^2(\mathbb{Z}^n)}\\
&=& {\rm Re}( \Op_{\mathbb{T}^n}(\tau)^* g,g)_{L^2(\mathbb{T}^n)}\\
&=& {\rm Re}( \Op_{\mathbb{T}^n}(\tau) g,g)_{L^2(\mathbb{T}^n)}  \\
&\geq & -C\|g\|_{H^{\frac{m-1}{2}}(\mathbb{T}^n)}\\
&=& -C\|f\|_{\ell^{2}_{\frac{m-1}{2}}(\mathbb{Z}^n)},
\end{eqnarray*}
finishing the proof. 
\end{proof}

 \subsection{Well-posedness of the parabolic equations}
 
As a consequence of G{\aa}rding inequalities on the lattice, we obtain the existence and uniqueness of solutions of the parabolic equation on the lattice.
\begin{thm}\label{THM:parabolic}
Let $A \in \Op_{\mathbb{Z}^n}{S}^{m}_{1,0}(\mathbb{Z}^n \times \mathbb{T}^n)$, $m>0$.
Assume that  there exist $C_0 >0$ and  $ d_0 >0$ such that for all $x\in\mathbb T^n$, we have
\begin{equation}\label{gard 5}
 -\sigma_A(k,x) \geq C_0|k|^m  \qquad \text{for} \ |k| \geq d_0.
 \end{equation}
Let $T>0$. Then for every $u_0\in \ell^2(\mathbb{Z}^n)$ and $f \in L^1([0,T],\ell^2(\mathbb{Z}^n))$, the equation
\begin{equation}\label{42}
\begin{cases}
\frac{\partial u}{\partial t} - Au &= f, \qquad t\in [0,T],\\
u(0)  &= u_0,
\end{cases}
\end{equation}  
has a unique solution $u \in C([0,T],\ell^2(\mathbb{Z}^n))$.
Moreover, there exists $C>0$ such that for all $t\in [0,T]$ we have
\begin{equation}\label{gard 6}
\|u(t)\|^{2}_{\ell^2(\mathbb{Z}^n)} \leq C\Big(\|u_0\|^{2}_{\ell^2(\mathbb{Z}^n)} +  \int^{t}_{0} \|f(\tau)\|^{2}_{\ell^2(\mathbb{Z}^n)} d\tau\Big).
 \end{equation}
\end{thm}
\begin{proof}
Suppose that $u$ satisfies  the condition \eqref{42}. We first observe that \eqref{gard 5} implies that there exists $C'_0 >0$ such that
\begin{equation*}
 |-(\sigma_A + \sigma_A^*)(k,x)| \geq C'_0|k|^m  \qquad \text{for} \ |k| \geq d_0.
 \end{equation*}
 Then by the G{\aa}rding inequality in Theorem \ref{THM:Garding} and Theorem \ref{THM:adjoint} about adjoint operators, we have  
 \begin{equation}\label{gard 7}
 -\Big( (A+A^*)u,u\Big)_{\ell^2(\mathbb{Z}^n)} \geq C_1||u||^{2}_{\ell^{2}_{\frac{m}{2}}(\mathbb{Z}^n)}-C_2||u||^{2}_{\ell^2(\mathbb{Z}^n)}.
 \end{equation}
 
 Then by \eqref{42} we have 
 \begin{eqnarray*}
 \frac{d }{dt} ||u||^{2}_{\ell^2(\mathbb{Z}^n)} &=&  \frac{d }{dt}\bigg(u(t),u(t)\bigg)_{\ell^2(\mathbb{Z}^n)}\\
  &=& \bigg( \frac{d }{dt},u\bigg)_{\ell^2(\mathbb{Z}^n)} + \bigg(u, \frac{d }{dt}\bigg)_{\ell^2(\mathbb{Z}^n)} \\
&=& \bigg(Au +f,u \bigg)_{\ell^2(\mathbb{Z}^n)} + \bigg(u,Au +f \bigg)_{\ell^2(\mathbb{Z}^n)}\\
&=& \bigg((A+A^*)u ,u \bigg)_{\ell^2(\mathbb{Z}^n)} + 2{\rm Re}(u,f)_{\ell^2(\mathbb{Z}^n)} \qquad \text{by \  \eqref{gard 7} \ we \ have}\\
&\leq & -C_1||u(t)||^{2}_{\ell^{2}_{\frac{m}{2}}(\mathbb{Z}^n)}+C_2||u(t)||^{2}_{\ell^2(\mathbb{Z}^n)} + ||u(t)||^{2}_{\ell^2(\mathbb{Z}^n)} + ||f||^{2}_{\ell^2(\mathbb{Z}^n)}\\
&\leq & (C_2 +1)||u(t)||^{2}_{\ell^2(\mathbb{Z}^n)} + ||f||^{2}_{\ell^2(\mathbb{Z}^n)}.
\end{eqnarray*}
 Thus by Gronwall's lemma we have that there exists $C >0$ such that 
 \begin{equation*}
\|u(t)\|^{2}_{\ell^2(\mathbb{Z}^n)} \leq C\Big(\|u_0\|^{2}_{\ell^2(\mathbb{Z}^n)} +  \int^{T}_{0} \|f(\tau)\|^{2}_{\ell^2(\mathbb{Z}^n)} d\tau\Big).
 \end{equation*}
 This proves \eqref{gard 6}.
 
 By the standard Picard iteration scheme it the follows that \eqref{42} has a solution $u \in C([0,T],\ell^2(\mathbb{Z}^n))$.
 
 Next we prove the uniqueness. For this, let $u,v$ be two solutions of \eqref{42}, that is 
 \begin{equation*}
 \begin{cases}
\frac{\partial u}{\partial t} - Au &= f, \; t\in [0,T],\\
u(0) &= u_0,
\end{cases}
 \end{equation*}
 \begin{equation*}
 \begin{cases}
\frac{\partial v}{\partial t} - Av &= f, \; t\in [0,T],\\
v(0) &= u_0.
\end{cases}
 \end{equation*}
 
By this, setting $w:= u-v$, we have
\begin{equation*}
 \begin{cases}
\frac{\partial w}{\partial t} - Aw &= 0, \; t\in [0,T],\\
w(0) &= 0.
\end{cases}
 \end{equation*} 
According to \eqref{gard 6}, we then must have $\|w(t)\|_{\ell^2(\mathbb{Z}^n)} = 0$ for all $t \in [0,T]$.
Therefore  this implies 
\[0=\|w(t)\|_{\ell^2(\mathbb{Z}^n)} = \|u(t)-v(t)\|_{\ell^2(\mathbb{Z}^n)}  \quad \implies  u(t) = v(t) \quad \text{for \ all} \;t\in [0,T],\]
completing the proof.
\end{proof}

\subsection{Boundedness and compactness on $\ell^p(\Zn)$}

   The following statement gives a condition for the $\ell^{p}(\Zn)$-boundedness of pseudo-differential operators on $\mathbb Z^n$ in terms of their symbols: here one asks for decay of irregular symbol; the regular case is given in \eqref{EQ:Lp-bound} below.
   Let
   \begin{equation}\label{EQ:Ftn}
(\mathcal{F}_{\mathbb T^n}\sigma)(k,m):=\int_\Tn e^{- 2\pi i m\cdot x }\sigma(k,x)\text{d}x.
\end{equation} 
   We note that it was shown in \cite{CR2011} that if $\mathcal{F}_{\mathbb T^n}\sigma\in \ell^q(\Zn\times\Zn)$ then $\Op(\sigma):\ell^p(\Zn)\to\ell^p(\Zn)$ is bounded provided that    $2\leq p<\infty$ and $\frac1p+\frac1q=1$.
   
   Moreover, if, in general, $\Op(\sigma):\ell^p(\Zn)\to\ell^r(\Zn)$ is bounded, then
   for every $m\in\Zn$ the function $(\mathcal{F}_{\mathbb T^n}\sigma)(k,m-k)$ must be in $\ell^r(\Zn)$ as a function of $k$. This follows since the latter condition is equivalent to saying that $\Op(\sigma)w_m\in \ell^r(\Zn)$ for all $m\in\Zn$, for functions $w_m$ such that $w_{m}(l)=\delta_{ml}$, in view of $(\mathcal{F}_{\mathbb T^n}\sigma)(k,m-k)=\Op(\sigma)w_m(k)$, see \eqref{EQ:www}.
   
   Also, it is known that for $1\leq p\leq\infty$, for amplitude operators with amplitude $a$, we have the following analogue of the Calderon-Vaillancourt theorem:
   \begin{equation}\label{EQ:Lp-bound}
\|\Op(a)\|_{\mathscr{L}(\ell^{p}(\mathbb{Z}^n))} \leq C\sup_{(k,l,x)\in\Zn\times\Zn\times\Tn, |\alpha|\leq n+1} |\partial_x^\alpha a(k,l,x)|,
\end{equation} 
 see  \cite[Theorem 2.8]{RR-AAM}.
   
   \begin{proposition}\label{compact} Let $1\leq p< \infty.$ 
   Let $\sigma:\mathbb{Z}^n\times \mathbb{T}^n\to\C$ be a measurable function. Assume that there is a positive constant $C>0$ and a function $\omega\in \ell^{1}(\mathbb Z^n)$ such that
   \[|(\mathcal{F}_{\mathbb T^n}\sigma)(k,m)| \leq C|\omega(m)|, \quad \textrm{ for all }\; k,m \in \mathbb{Z}^n,\]
   where $\mathcal{F}_{\mathbb T^n}\sigma$ is the Fourier transform of $\sigma$ in the second variable.
   Then $\Op(\sigma): \ell^{p}(\mathbb{Z}^n)\rightarrow \ell^{p}(\mathbb{Z}^n)$ is a bounded linear operator and 
   \[\|\Op(\sigma)\|_{\mathscr{L}(\ell^{p}(\mathbb{Z}^n))} \leq C\|\omega\|_{\ell^{1}(\mathbb{Z}^n)}. \]
    \end{proposition}
   
   The proof of this result is straightforward once we observe that the assumption means that the convolution kernel of $\Op(\sigma)$ is in $\ell^1$, and then the statement follows by the Young inequality. We give a simple argument for completeness and also to prepare for Theorem \ref{THM:cpt}.
   
   \begin{proof}[Proof of Proposition \ref{compact}]
   Let $f\in \ell ^{1}(\mathbb Z^n).$
   We can write the operator $\Op(\sigma)$ as
   \begin{eqnarray*}
   \text{Op}(\sigma)f(k) &=& \sum_{m\in \mathbb{Z}^{n}}f(m)\int_{\mathbb{T}^n}e^{-2\pi i(m-k)\cdot x}\sigma(k,x)\text{d}x \\
   &=& \sum_{m\in \mathbb{Z}^{n}}f(m)(\mathcal{F}_{\mathbb T^n}\sigma)(k,m-k).
   \end{eqnarray*}
   Let us define $(\mathcal{F}_{\mathbb T^n}\sigma)^{\sim}$ by 
   \[(\mathcal{F}_{\mathbb T^n}\sigma)^{\sim}(k,m) = (\mathcal{F}_{\mathbb T^n}\sigma)(k,-m). \]
   It follows that we can write $\text{Op}(\sigma)f$ as a convolution
   \begin{eqnarray*}
   \text{Op}(\sigma)f(k)&=& \sum_{m\in \mathbb{Z}^{n}}f(m)(\mathcal{F}_{\mathbb T^n}\sigma)^{\sim}(k,k-m) \\
   &=& ((\mathcal{F}_{\mathbb T^n}\sigma)^{\sim}(k,\cdot) * f)(k).
   \end{eqnarray*}
   Taking absolute value to the power $p$ and the sum of both sides, we obtain 
   \begin{eqnarray*}
   \|\text{Op}(\sigma)f\|_{\ell^{p}(\mathbb{Z}^n)}^{p} &=& \sum_{k\in \mathbb{Z}^{n}}|((\mathcal{F}_{\mathbb T^n}\sigma)^{\sim}(k,\cdot) * f)(k) |^{p}\\
   &\leq & \sum_{k\in \mathbb{Z}^{n}}((|(\mathcal{F}_{\mathbb T^n}\sigma)^{\sim}(k,\cdot)| * |f|)(k)) ^{p}. \\
   &\leq &
    C^p \sum_{k\in \mathbb{Z}^{n}}\Big((|\omega| * |f|)(k)\Big)^{p} \\
    &\leq & C^p \|\omega\|_{\ell ^{1}(\mathbb{Z}^n)}^{p}\|f\|_{\ell^{p}(\mathbb{Z}^n)}^{p},
   \end{eqnarray*}
   using Young's inequality for convolution in the last line. 
      The fact that  $\ell^{1}(\mathbb Z^n)$ is dense in $\ell^{p}(\mathbb Z^n)$ completes the proof for all $1\leq  p< \infty$.
   \end{proof}
   
 For $n=1$, these statements were established in \cite{molahajloo2009pseudo}. 
   
   One condition for compactness of operators appeared in Corollary \ref{COR:comp}. Now we record another condition, strengthening the condition of Theorem \ref{compact} on the symbol $\sigma$ to guarantee that the corresponding pseudo-difference operator is compact on  $\ell ^{p}(\mathbb Z^n)$.
      
   \begin{thm}\label{THM:cpt}
   Let $\sigma:\mathbb{Z}^{n}\times \mathbb{T}^n\to\C$ be a measurable function  such that there exist a positive function $\lambda:\mathbb Z^n\to\R$ and a function $\omega\in \ell ^{1}(\mathbb Z^n)$ such that
      \[|(\mathcal{F}_{\mathbb{T}^n}\sigma)(k,m)| \leq \lambda(k)|\omega(m)|, \quad\textrm{ for all }\; m, k \in \mathbb Z^n,\]
   and such that
   \[\lim\limits_{|k|\rightarrow \infty}\lambda(k) = 0.\]
   Then the pseudo-difference operator $\Op(\sigma): \ell^p(\mathbb{Z}^n)\rightarrow \ell ^p(\mathbb{Z}^n)$ is a compact operator for all $1\leq p<\infty.$
   \end{thm}
   \begin{proof}
  Let us consider the sequence of functions 
    \[\sigma_{N}(k,x) := \left\{\begin{aligned}
   \sigma(k,x) , \quad |k|\leq N,\\
   0, \quad |k|> N. \\
   \end{aligned}\right.\]
   Then we have
   \begin{equation}\label{sub}
\begin{aligned}
\big(\text{Op}(\sigma) - \text{Op}(\sigma_N)\big)f(k) & = \int_{\mathbb T^n}e^{2\pi ik\cdot x}(\sigma - \sigma_N)(k,x)\widehat{f}(x)\text{d}x \\
&  = \sum_{m\in \mathbb Z^n}f(m)\int_{\mathbb T^n}e^{-2\pi i(m-k)\cdot x}(\sigma - \sigma_N)(k,x)\text{d}x \\
&   = \sum_{m\in \mathbb Z^n}f(m)(\mathcal{F}_{\mathbb T^n}(\sigma - \sigma_N))(k,m-k).
\end{aligned}
\end{equation} 
  Taking the $\ell^p$-norm and writing this using the representation as a convolution we get
   \begin{eqnarray*}
   \|\big(\text{Op}(\sigma) - \text{Op}(\sigma_N)\big)f\|_{\ell^p(\mathbb Z^n)}^{p}
   & \leq &  \sum_{k\in\mathbb Z^n}\Bigg(\Big(\big|\big(\mathcal{F}_{\mathbb T^n}(\sigma - \sigma_N)\big)^{\sim}(k,\cdot)\big|\ast\big| f\big|\Big)(k)\Bigg)^{p}\\ 
   &\leq & \sum_{|k|>N}\Bigg(\Big(\big|\big(\mathcal{F}_{\mathbb T^n}\sigma\big)^{\sim}(k,\cdot)\big|\ast\big| f\big|\Big)(k)\Bigg)^{p}.
   \end{eqnarray*}
   By hypothesis we have that for every $\varepsilon>0$  there exists some $N_0$ such that 
   $|\lambda(k)|<\varepsilon$, for all $k>N_0$, and hence also
     \begin{equation*}
   |(\mathcal{F}_{\mathbb{T}^n\sigma})^{\sim}(k,m)|^p \leq \varepsilon^p|\omega(m)|^p.
   \end{equation*}
   Using this and the Young inequality for convolutions for $N>N_0$ we obtain
   \begin{eqnarray*}
   \|\big(\text{Op}(\sigma) - \text{Op}(\sigma_N)\big)f\|_{\ell^p(\mathbb Z^n)}^{p} &\leq &  \sum_{|k|>N}\Bigg(\Big(\varepsilon\big|\omega\big|\ast\big| f\big|\Big)(k)\Bigg)^{p}\\
   &=& \varepsilon^p\|\omega\ast f\|_{\ell^{p}(\mathbb Z^n)}^{p}\\
   &\leq &  \varepsilon^{p}\|\omega\|_{\ell^{1}(\mathbb Z^n)}^{p}\|f\|_{\ell^{p}(\mathbb Z^n)}^{p}.
   \end{eqnarray*}
   Using the density of $\ell ^1(\mathbb{Z}^n)$ in $\ell ^p(\mathbb{Z}^n)$ we obtain
   \[\|\text{Op}(\sigma) - \text{Op}(\sigma_N)\|_{\mathscr{L}(\ell^p(\mathbb Z^n))} \leq  \varepsilon\|\omega\|_{\ell^1(\mathbb Z^n)}.\]
  It implies that $\text{Op}(\sigma)$ is the limit in norm of a sequence of compact operator on $\ell^{p}(\mathbb{Z}^n)$, therefore $\text{Op}(\sigma)$ is $\ell^{p}$-compact.
   \end{proof}
  
\subsection{Fourier series operators}

  The same argument as in the proof of Theorem \ref{THM:L2} allows one to extend it to a more general setting of Fourier series operators. Before we formulate a result let us introduce some notation.
   
Let $\psi:\Rn\times\Zn\to\R$ be a real-valued function such that function
$x\mapsto{e}^{i \psi(x,k)}$ is $1$-periodic for
every $k\in\Zn$. In this case, by abuse of notation, we can still write $x\in\Tn$.
For $\tau:\Tn\times\Zn\to\C$ and $v\in C^\infty(\Tn)$ let us define the operator
$T_\Tn(\psi,\tau)$
by 
\begin{equation}\label{EQ:Tphitau}
T_\Tn(\psi,\tau) v(x):=\sum_{k\in\Zn} e^{i \psi(x,k)} \tau(x,k)
\p{\Ftn{v}}(k).
\end{equation} 
Properties of such operators and their extensions have been extensively analysed in 
\cite[Section 9]{RT-JFAA} and in \cite[Sections 4.13-4.15]{ruzhansky2009pseudo}, to which we refer for their calculus, boundedness properties, and applications to hyperbolic equations.

Analogously, let $\phi:\Zn\times\Rn\to\R$ be a real-valued function such that function
$x\mapsto{e}^{i \phi(k,x)}$ is $1$-periodic for
every $k\in\Zn$.
For $\sigma:\Zn\times \Tn\to\C$ and $f\in \mathcal{S}(\Zn)$ let us define the operator
$T_\Zn(\phi,\sigma)$
by 
\begin{equation}\label{EQ:Tphitau}
T_\Zn(\phi,\sigma) f(k):=\int_\Tn e^{i \phi(k,x)} \sigma(k,x)
\p{\Fzn{f}}(x)\text{d}x.
\end{equation} 
In the special case of $\phi(k,x)=2\pi k\cdot x$ we have $T_\Zn(\phi,\sigma)=\Opzn(\sigma)$, so in analogy to $T_\Tn(\psi,\tau)$ we may call operators $T_\Zn(\phi,\sigma)$ Fourier series operators.

 \begin{thm}\label{THM:fsos}
Let $\phi:\Zn\times\Rn\to\R$ be a real-valued function such that function
$x\mapsto{e}^{i \phi(k,x)}$ is $1$-periodic for
every $k\in\Zn$, and let $\sigma:\Zn\times \Tn\to\C$.
\begin{itemize}
\item[(i)] Define $\tau(x,k):=\overline{\sigma(-k,x)}$ and $\psi(x,k):=-\phi(-k,x)$.
Then we have
\begin{equation}\label{EQ:link4}
T_\Zn(\phi,\sigma)=\Fzn^{-1}\circ T_\Tn(\psi,\tau)^*\circ \Fzn,
\end{equation} 
where $T_\Tn(\psi,\tau)^*$ is the adjoint of the operator $T_\Tn(\psi,\tau).$
\item[(ii)] Assume  that 
for all $|\alpha|\leq 2n+1$ and $|\beta|=1$ we have
\begin{equation}\label{l2a1}
  \left| \partial_x^\alpha \sigma(k,x) \right| \leq
  C \textrm{ and }
  \left| \partial_x^\alpha\triangle_k^\beta \phi(k,x) \right| \leq
  C\;\; \textrm{ for all } (k,x)\in \Zn\times\Tn.
\end{equation}
Assume also that 
\begin{equation}\label{l2a2}
  \left| \nabla_x\phi(k,x)-\nabla_x\phi(l,x) \right| \geq
  C|k-l|\; \textrm{ for all } x\in\Tn, \; k,l\in\Zn.
\end{equation}
Then $T_\Zn(\phi,\sigma)$ extends to a bounded operator 
on $\ell^{2}(\Zn)$.
\end{itemize} 
 \end{thm} 
 
 Part (i) follows by the same argument as that in the proof of Theorem \ref{THM:link}, so we omit the details. Part (ii) follows by the same argument as that in the proof of Theorem \ref{THM:L2}, with the exception that instead of the $L^2$-boundedness of toroidal pseudo-differential operators we use the $L^2$-boundedness of the toroidal Fourier series operators as in \cite[Theorem 9.2]{ruzhansky2009pseudo}, see also \cite[Theorem 4.14.2]{ruzhansky2009pseudo}.

 % % % % % % % % % % % % % % % % % % % % % % % % % % % % % % % % % % % % % % % % % % % % % % % % % % % % % %
% \section*{Appendix}
  \section{Examples}
  \label{SEC:Ex}
  
 Let us give some examples of operators and their symbols as well as applications to solutions of difference equations, as an example of applications of our constructions.
 
 Let $v_{j} = (0, \ldots, 0, 1, 0, \ldots, 0)\in \mathbb{Z}^n$, where $1$ is the $j^{th}$ element of $v_j$.  
    
    \begin{enumerate}
    \item Consider the operator $A_j$ defined by 
    \[A_{j}f(k) = f(k+v_j) - f(k).\]
    Defining $e_x(k)= e^{2\pi ik\cdot x}$ for all $k\in\Zn$ and $x\in \mathbb{T}^n$, we have
    \[A_{j}e_{x}(k) = e^{2\pi i(k + v_j)\cdot x} - e^{2\pi i k\cdot x},\] 
    hence by Proposition \ref{PROP:symbols} the symbol of $A_j$ is given by
    \[\sigma_{A_{j}}(k,x) = e^{2\pi i v_j\cdot x} - 1 = e^{2\pi i x_j} - 1.\]
    The symbol $\sigma_{A_{j}}$ is independent of $k$ and $\sigma_{A_{j}}\in S^{0}(\mathbb{Z}^n\times \mathbb{T}^n)$. Moreover, the symbol $\sigma_{A_{j}}$ is not elliptic.

    \item  The operator $B_j$ defined by 
    \[B_{j}f(k) = |k|^{\mu}(f(k+v_j)+1)-|k|^\nu (f(k-v_j)+2)\]
    has symbol 
    \[\sigma_{B_{j}}(k,x) = |k|^{\mu}(e^{2\pi i x_j}+1)-|k|^{\nu}(e^{-2\pi i x_j}+2)\in S^{\max\{\mu,\nu\}}(\mathbb{Z}^n \times \mathbb{T}^n),\]
    which is  elliptic of order $\nu$ if, for example, $\nu\geq \mu$. It is not elliptic if $\mu>\nu$.
    It follows from Corollary \ref{COR:L2} that if 
    $$|k|^{\mu}(f(k+v_j)+1)-|k|^\nu (f(k-v_j)+2)=g(k),\quad\textrm{ for all }\; k\in\Zn,$$
    as well as $\nu\geq \mu$ and $g\in\ell^2_s(\Zn)$ then $f\in \ell^2_{s+\nu}(\Zn)$ for all $s\in\R$,
    where $\ell^2_s(\Zn)$ is the weighted space defined in \eqref{EQ:l2s}.
    
    \item Let us define the operator $T$ by 
    \[{T}f(k) := \sum_{j=1}^{n}\Big(f(k + v_j) - f(k - v_j)\Big) + af(k).\]
    It has symbol 
    $$
    \sigma_{T}(k,x) = \sum_{j=1}^{n}\Big(e^{2\pi i x_j} - e^{-2\pi i x_j}\Big) + a
    =2i\sum_{j=1}^{n} \sin(2\pi x_j) + a 
    $$
    in $S^0(\Zn\times\Tn)$, 
    which is elliptic if ${\rm Re }\, a \neq 0$ or if ${\rm Im }\, a\not\in [-2n,2n].$ Consequently, in these cases the operator inverse $T^{-1}\in \text{Op}(S^{0}(\mathbb{Z}^n\times \mathbb T^n))$ has symbol 
     \[\sigma_{T^{-1}}(x) = \frac{1}{2i\sum_{j=1}^{n} \sin(2\pi x_j)  + a}, \quad x\in \mathbb{T}^n.\]
    Hence the inverse operator of $T$ is given by 
    \[T^{-1}g(k) = \int_{\mathbb{T}^n}e^{2\pi k\cdot x}\frac{1}{2i\sum_{j=1}^{n} \sin(2\pi x_j)  + a} \widehat{g}(x)\text{d}x,\] 
    solving the equation
    \begin{equation}\label{EQ:exeq}
\sum_{j=1}^{n}\Big(f(k + v_j) - f(k - v_j)\Big) + af(k)=g(k).
\end{equation} 
    By Corollary \ref{COR:L2} the operator $T^{-1}$ is bounded from  
   $\ell^2_{s}(\Zn)$ to $\ell^2_{s}(\Zn)$ for any $s\in\R$ that is, if $g\in \ell^2_{s}(\Zn)$ then the solution $f$ to \eqref{EQ:exeq} satisfies 
 $f\in \ell^2_{s}(\Zn)$.
    \end{enumerate}
   
%\bibliographystyle{alphaabbr}
%\bibliography{Gael-Linda}
%\end{document}

 \end{document}